\documentclass[10pt, reqno]{amsart}
\usepackage{amsmath, amsthm, amscd, amsfonts, amssymb, graphicx, color, mathrsfs}
\usepackage[bookmarksnumbered, colorlinks, plainpages]{hyperref}
\usepackage{amssymb}
\usepackage{cite}
\usepackage[all]{xy}
\usepackage{slashed}
\usepackage{tikz-cd}
\usepackage{mathabx}
\usepackage{tipa}
\usepackage{soul}
\usepackage{cancel}
\usepackage{ulem}
\usepackage{algorithm}
\usepackage{algpseudocode}
\usepackage{enumitem}
\usepackage{lipsum}
\usepackage{blindtext}
\usepackage{overpic}

\usepackage{todonotes}
\usepackage{comment}

\DeclareMathOperator{\Ad}{Ad}
\DeclareMathOperator{\Span}{span}

\allowdisplaybreaks

\newlength\tindent
\newtheorem{theorem}{Theorem}[section]
\newtheorem{lemma}[theorem]{Lemma}

\newtheorem{proposition}[theorem]{Proposition}
\newtheorem{corollary}[theorem]{Corollary}

\theoremstyle{definition}
\newtheorem{definition}[theorem]{Definition}

\theoremstyle{remark}
\newtheorem{remark}[theorem]{Remark}
\numberwithin{equation}{section}

\begin{document}
\setcounter{page}{1}

\title[Riemannian geometry of $G_2$-type real flag manifolds]{Riemannian geometry of $G_2$-type real flag manifolds}

\author[B. Grajales]{Brian Grajales}
\address{Brian Grajales \endgraf
IMECC-Unicamp, Departamento de Matemática, Rua Sérgio Buarque de Holanda
651, Cidade Universitária Zeferino Vaz. 13083-859, Campinas - SP, Brazil
\endgraf
  {\it E-mail address} {\rm grajales@ime.unicamp.br}
  }

\author[G. Rondón]{Gabriel Rondón}
\address{Gabriel Rondón \endgraf 
S\~{a}o Paulo State University (Unesp), Institute of Biosciences, Humanities and
	Exact Sciences. Rua C. Colombo, 2265, CEP 15054--000. S. J. Rio Preto, S\~ao Paulo,
	Brazil}
\email{gabriel.rondon@unesp.br}

\author[J. Saavedra]{Julieth Saavedra}
\address{Julieth Saavedra \endgraf
	Centro de Ciências - Departamento de Matemática, Av. Humberto Monte, s/n, Campus do Pici - Bloco 914. 60455-760, Fortaleza, Ceará, Brazil}
\email{julieth.p.saavedra@gmail.com}


     \keywords{Real flag manifold, homogeneous metrics, split real form, Ricci flow}
     \subjclass[2020]{22F30, 53C30}

\begin{abstract} In this paper, we investigate homogeneous Riemannian geometry on real flag manifolds of the split real form of $\mathfrak{g}_2$. We characterize the metrics that are invariant under the action of a maximal compact subgroup of $G_2$. Our exploration encompasses the analysis of g.o. metrics and equigeodesics on the $\mathfrak{g}_2$-type flag manifolds. Additionally, we explore the Ricci flow for the case where the isotropy representation has no equivalent summands, employing techniques from the qualitative theory of dynamical systems.        
\end{abstract} 
\maketitle

\tableofcontents
\allowdisplaybreaks

\section{Introduction}
In the context of homogeneous manifolds, the study of real flag manifolds stands out as a fascinating pursuit, engaging in the intricacies of geometric structures and their underlying symmetries. A real flag manifold is a quotient space $\mathbb{F}=G/P,$ where $G$ is a connected Lie group with non-compact real simple associated Lie algebra $\mathfrak{g},$ and $P$ is a parabolic subgroup of $G.$ In this paper, we consider the case where $\mathfrak{g}$ is the split real form of $\mathfrak{g}_2.$ We present an exploration of the homogeneous Riemannian geometry in these manifolds, focusing particularly on homogeneous geodesics and the Ricci flow of invariant metrics.\\

For a real flag manifold $\mathbb{F}=G/P,$ we have that any maximal compact subgroup $K\subseteq G$ acts transitively on $\mathbb{F}$ with isotropy subgroup $K\cap P.$ This leads to an alternative presentation of $\mathbb{F},$ namely, $\mathbb{F}=K/(K\cap P).$ This presentation yields the isotropy representation of $K\cap P$ on the tangent space $T_o\mathbb{F}$ at the point $o:=e(K\cap P),$ where $e$ denotes the identity element of $K.$ The compactness of $K$ ensures the complete decomposition of this representation into irreducible subrepresentations. The understanding of these subrepresentations and their relations is key to describing $K$-invariant tensor fields on $\mathbb{F}.$ A detailed study of the isotropy representation of a real flag manifold was developed in \cite{PSM}. Several authors have contributed to the study of geometry and topology of real flag manifolds (see, for instance, \cite{SGE,GG,GGN,Mats,Varea}). In particular, we point out \cite{GGN}, where the authors described the invariant metrics on flag manifolds of a split real form of a classical Lie algebra ($A,$ $B,$ $C,$ or $D$) and used this description to characterize those invariant metrics for which every geodesic through the origin $o$ is a homogeneous geodesic. This means that every geodesic is the orbit of a one-parameter subgroup of $G.$ An invariant metric with that property is called a g.o. metric (geodesic orbit metric), and the corresponding Riemannian homogeneous space is called a g.o. space. This class of homogeneous spaces includes simply connected symmetric spaces and naturally reductive homogeneous spaces. While a complete classification of g.o. spaces is far from being accomplished, there exists a substantial body of literature on this matter; we refer to \cite{aa,CNN,CZZ,N} for instance.\\

The primary objective of this paper is to continue the aforementioned work in \cite{GGN} by providing a detailed description of the invariant metrics on real flag manifolds associated with the exceptional Lie algebra $\mathfrak{g}_2$ and classifying the g.o. metrics among them (Theorem \ref{g.o.metrics:theorem}). However, our scope extends beyond this; we also aim to characterize homogeneous curves (orbits of a one-parameter subgroup) that are geodesics for every invariant metric (Theorem \ref{equigeodesic:theorem}). These special curves are known as equigeodesics. Recent works on the classification and properties of these curves include \cite{CGN,GGEquigeodesics,Stat}. Here, we mainly use the results provided in \cite{GGEquigeodesics} to obtain the equigeodesics on real flag manifolds of $\mathfrak{g}_2.$\\

As a final contribution, we explore the dynamics of the system associated with the homogeneous Ricci flow in the case where all the irreducible subrepresentations of the isotropy representation have multiplicity one. The Ricci flow is currently one of the most studied topics in Differential Geometry. For a differentiable manifold $M,$ it is defined as the nonlinear evolution equation 
\begin{equation}\label{Ricci:flow:0}
    \dfrac{\partial g}{\partial t}=-2\textnormal{Ric}(g),
\end{equation}
where $t\mapsto g(t)$ is a one-parameter family of Riemannian metrics on $M$ and $\textnormal{Ric}(g)$ is the Ricci tensor associated with $g.$ Hamilton introduced it in \cite{Hamilton}, gaining significance due to its implications for understanding the geometric and topological structure of Riemannian manifolds. In the case of a homogeneous space, the Ricci tensor is constant on $M,$ and each solution $g$ of \eqref{Ricci:flow:0} is a curve on the set of invariant metrics, provided the initial condition $g(0)$ is invariant. Consequently, the equation \eqref{Ricci:flow:0} transforms into an ordinary differential equation known as the homogeneous Ricci flow. While this makes it somewhat more manageable, it remains far from straightforward. Exploring the Ricci flow on homogeneous spaces involves tools from the theory of dynamical systems. This approach has been adopted by various works covering different classes of homogeneous spaces, including generalized Wallach spaces \cite{AANS,StathaRicci}, Stiefel manifolds \cite{StathaRicci}, and complex flag manifolds \cite{GM,GMPSS,StathaRicci}.\\

Notably, in \cite{GM}, the authors employed the Poincaré compactification method \cite{10.2307/2001320} for the first time to study the global behavior of the homogeneous Ricci flow on $\frac{\textnormal{SO}(2n+1)}{\textnormal{U}(m)\times\textnormal{SO}(2k+1)}$ and $\frac{\textnormal{Sp}(n)}{\textnormal{U}(m)\times\textnormal{Sp}(k)}.$ This tool has proven to be very useful, and we will utilize it Section \ref{sec_2} for the analysis of the global dynamics. Our approach for the study of the homogenous Ricci flow consists of a local study and a global one. Concerning the local study, we calculate the invariant algebraic surfaces of degree 2 and since said system is very degenerate we use the Blow-up method \cite{DurRou1996} to understand the local dynamics of the system around $z-$axis, which consists of changing an equilibrium point, whose Jacobian matrix has eigenvalues with zero real part, for a sphere $\mathbb{S}^2\subset\mathbb{R}^3$, leaving the dynamics far from this point without changes.
For the analysis of global dynamics, as previously mentioned, we employ the classical Poincaré compactification method. This method involves identifying $\mathbb{R}^3$ with the interior of the unit sphere $\mathbb{S}^2$ and $\mathbb{S}^2$ with the infinity of $\mathbb{R}^3$. Subsequently, the polynomial differential system defined in $\mathbb{R}^3$ is analytically extended to the entire sphere. Consequently, we can examine the dynamics of polynomial differential systems in the vicinity of infinity.\\


The paper is organized in the following form. In Section \ref{sec:preliminares}, we present some basic results on compact homogeneous spaces, the non-compact Lie algebra $\mathfrak{g}_2$ and real flag manifolds of $\mathfrak{g}_2$. In particular, we present a description of the isotropy representation for each flag manifold. In Section \ref{sec:hom_geo}, we describe the g.o. metrics  and the equigeodesics on flag manifolds of $\mathfrak{g}_2.$ Finally, in Section \ref{sec:Ricci}, we perform an analysis on the homogeneous Ricci flow for the flag manifold whose isotropy representation has no equivalent irreducible subrepresentations.

\section{Preliminaries}\label{sec:preliminares}
\subsection{Compact homogeneous spaces}\label{section:2.1} Let $G$ be a compact Lie group, $H$ a closed subgroup of $G,$ and consider the homogeneous space $M=G/H.$ There is a natural smooth transitive action of $G$ on $G/H$ given by \begin{equation*}\label{action}\phi:G\times G/H\longrightarrow G/H;\ \phi(a,bH)=abH.\end{equation*} For each $a\in G,$ we can define the map $$\phi_a:=\phi(a,\cdot):G/H\ni bH\mapsto abH\in G/H.$$ A Riemannian metric $g$ on $G/H$ is called $G$-invariant (or $G$-homogeneous) if $$\{\phi_a:G/H\rightarrow G/H\ |\ a\in G\}\subseteq \textnormal{Iso}(G/H,g),$$ where $\textnormal{Iso}(G/H,g)$ denotes the group of all bijective isometries from $G/H$ to itself. \\

Let $\mathfrak{g}$ and $\mathfrak{h}$ be the Lie algebras associated with $G$ and $H$ respectively. Consider the adjoint representation $\Ad:G\rightarrow \textnormal{GL}(\mathfrak{g})$ of $G.$ Since $G$ is compact, there exists a unique (up to re-scaling) $\Ad(G)$-invariant inner product $(\cdot,\cdot)$ on $G.$ This means that $$(\Ad(a)X,\Ad(a)Y)=(X,Y),\ a\in G,\ X,Y\in\mathfrak{g}.$$ By fixing this inner product, we obtain a reductive orthogonal decomposition of the Lie algebra $\mathfrak{g}$ as follows: if $\mathfrak{m}$ is the orthogonal complement of $\mathfrak{h}$ in $\mathfrak{g}$ with respect to $(\cdot,\cdot),$ then  $$\mathfrak{g}=\mathfrak{h}\oplus\mathfrak{m},\ \textnormal{and}\ \Ad(h)\mathfrak{m}=\mathfrak{m},\ \forall h\in H.$$ This allows us to define the representation \begin{equation}\label{adjoint:representation}
\Ad^H\big{|}_{\mathfrak{m}}:H\rightarrow\textnormal{GL}(\mathfrak{m});\ \Ad^H\big{|}_{\mathfrak{m}}(h):=\Ad(h)\big{|}_{\mathfrak{m}},\ \forall h\in H,
\end{equation} which is equivalent to the isotropy representation of $G/H$ at the left coset $eH$ of the identity element $e\in G.$ By compactness of $H$ and the fact that $(\cdot,\cdot)$ is $\Ad(G)$-invariant, we have that this representation is completely reducible into pairwise $(\cdot,\cdot)$-orthogonal irreducible $H$-submodules, that is, \begin{equation}\label{submodules}\mathfrak{m}=\mathfrak{m}_1\oplus\cdots\oplus\mathfrak{m}_s,\end{equation} where $\Ad(h)\mathfrak{m}_j=\mathfrak{m}_j,\ \forall h\in H,$ and the representation $$\Ad^H\big{|}_{\mathfrak{m}_j}:H\rightarrow\textnormal{GL}(\mathfrak{m}_j);\ \Ad^H\big{|}_{\mathfrak{m}_j}(h):=\Ad(h)\big{|}_{\mathfrak{m}_j}$$ is irreducible for each $j\in\{1,...,s\}.$ The $H$-submodules $\mathfrak{m}_1,...,\mathfrak{m}_s$ are called the isotropy summands of the representation \eqref{adjoint:representation}.  Two isotropy summands $\mathfrak{m}_i$ and $\mathfrak{m}_j$ are equivalent if the representations $\Ad^H\big{|}_{\mathfrak{m}_i}$ and $\Ad^H\big{|}_{\mathfrak{m}_j}$ are equivalent. An inner product $\left\langle\cdot,\cdot\right\rangle:\mathfrak{m}\times\mathfrak{m}\rightarrow\mathbb{R}$ is called $\Ad(H)$-invariant if it satisfies the equation $$\langle\Ad(h)X,\Ad(h)Y\rangle=\langle X,Y\rangle,\ h\in H,\ X,Y\in\mathfrak{m}.$$ There exists a bijection between the set of all Riemannian $G$-invariant metrics on $G/H$ and the set of all $\Ad(H)$-invariant inner products on $\mathfrak{m}.$ Since $(\cdot,\cdot)$ is $\Ad(G)$-invariant, then  $(\cdot,\cdot)\big{|}_{\mathfrak{m}\times\mathfrak{m}}:\mathfrak{m}\times\mathfrak{m}\rightarrow \mathbb{R}$ is $\Ad(H)$-invariant. Consequently, for any $\Ad(H)$-invariant inner product $\langle\cdot,\cdot\rangle,$ there exists a linear  operator $A:\mathfrak{m}\rightarrow \mathfrak{m}$ such that $$\langle X,Y\rangle=(AX,Y),\ \forall X,Y\in\mathfrak{m}.$$ 
The operator $A$ is referred to as the metric operator associated with $\langle\cdot,\cdot\rangle.$ As both $\langle\cdot,\cdot\rangle$ and $(\cdot,\cdot)\big{|}_{\mathfrak{m}\times\mathfrak{m}}$ are $\Ad(H)$-invariant inner products, the operator $A$ is positive definite, self-adjoint (with respect to $(\cdot,\cdot)\big{|}_{\mathfrak{m}\times\mathfrak{m}}$), and commutes with $\Ad(h)\big{|}_{\mathfrak{m}},$ for all $h\in H.$ Moreover, any linear operator $A$ that satisfies these properties corresponds to the metric operator associated with some $\Ad(H)$-invariant inner product on $\mathfrak{m}.$\\

Following the notations introduced in \cite{GGEquigeodesics}, let $\{T_i^j:i,j\in\{1,...,s\}\}$ be a family of linear maps $T_i^j:\mathfrak{m}_i\rightarrow\mathfrak{m}_j$ that satisfies the following properties:
\begin{itemize}
    \item[$i)$] $T_i^i=\textnormal{I}_{\mathfrak{m}_i},\ i=1,...,s.$\\
    
    \item[$ii)$] $T_i^j=0$ whenever $\mathfrak{m}_i$ is not equivalent to $\mathfrak{m}_j.$\\

    \item[$iii)$] If $\mathfrak{m}_i$ is equivalent to $\mathfrak{m}_j,$ then $T_i^j:\mathfrak{m}_i\rightarrow\mathfrak{m}_j$ is an equivariant map such that $$(T_i^j(X),T_i^j(Y))=(X,Y).$$

    \item[$iv)$] If $\mathfrak{m}_i$ is equivalent to $\mathfrak{m}_j,$ then $(T_i^j)^{-1}=T_j^i.$
\end{itemize}

\

There exist $(\cdot,\cdot)\big{|}_{\mathfrak{m}\times\mathfrak{m}}$-orthonormal sets $\mathcal{B}_1,...,\mathcal{B}_s$ such that for each $j\in\{1,...,s\},$ $\mathcal{B}_j$ is a basis of $\mathfrak{m}_j,$ and $T_i^j(\mathcal{B}_i)=\mathcal{B}_j$ whenever $\mathfrak{m}_i$ is equivalent to $\mathfrak{m}_j.$ The set $\mathcal{B}=\mathcal{B}_1\cup\cdots\cup\mathcal{B}_s$ is then a basis of $\mathfrak{m}.$ Any metric operator $A$ associated with a $G$-invariant metric on $G/H$ can be represented in such a basis by a matrix of the form
\begin{equation}\label{positive:matrix}
        \left[A\right]_{\mathcal{B}}=\left(\begin{array}{ccccc}\mu_1\textnormal{I}_{d_1} & A_{21}^T & A_{31}^T & \cdots & A_{s1}^T \\
        A_{21} & \mu_2\textnormal{I}_{d_2} & A_{32}^T & \cdots & A_{s2}^T\\
        A_{31} & A_{32} & \mu_3\textnormal{I}_{d_3} & \cdots & A_{s3}\\
        \vdots & \vdots & \vdots & \ddots & \vdots\\
        A_{s1} & A_{s2} & A_{s3} & \cdots &\mu_{s}\textnormal{I}_{d_s}
        \end{array}\right),
    \end{equation}
where $\mu_1,...,\mu_s>0,$ $d_i=\dim\mathfrak{m}_i,\ i=1,...,s,$ and $A_{ij}$ defines an equivariant map from $\mathfrak{m}_j$ to $\mathfrak{m}_i$ for $1\leq j<i\leq s$ (in particular, $A_{ij}=0$ whenever $\mathfrak{m}_i$ is not equivalent to $\mathfrak{m}_j$). Conversely, the formula \eqref{positive:matrix} defines a metric operator corresponding with some $G$
-invariant metric provided that the matrix is positive definite.

\subsection{The non-compact Lie algebra $\mathfrak{g}_2$} In this section, we present a construction of the split real form of the Lie algebra $\mathfrak{g}_2$. This construction can also be found in \cite{BM} or \cite[Section 8.4.1]{SM}.\\

Let $\mathfrak{sl}(3)$ represent the Lie algebra of real, traceless $3\times 3$ matrices, and $\mathbb{R}^3$ denote the 3-dimensional Euclidean space. The canonical basis of $\mathbb{R}^3$ is given by $\{e_1, e_2, e_3\}$ and its corresponding dual basis is denoted as $\{\epsilon_1, \epsilon_2, \epsilon_3\}$. We define $\bigwedge^3(\mathbb{R}^3)$ as the vector space consisting of all 3-covectors in $\mathbb{R}^3$. This vector space is isomorphic to $\mathbb{R}$ via the mapping $\Psi$ defined as follows:
$$\Psi:c\in\mathbb{R}\mapsto c\nu\in\bigwedge\nolimits^3(\mathbb{R}^3),\ \textnormal{where}\ \nu:=e_1\wedge e_2\wedge e_3.$$
Let us introduce the linear isomorphisms 
$$T:\bigwedge\nolimits^2(\mathbb{R}^3)\longrightarrow (\mathbb{R}^3)^*,\ T(u\wedge v)(w):=\Psi^{-1}(u\wedge v\wedge w) $$
and 
$$S:\bigwedge\nolimits^2(\mathbb{R}^3)^*\longrightarrow\mathbb{R}^3,$$
defined implicitly by the formula $$\alpha\wedge\beta\wedge\gamma=\gamma(S(\alpha\wedge\beta))\nu^*,\ \textnormal{where}\ \nu^*:=\epsilon_1\wedge\epsilon_2\wedge\epsilon_3,\ \textnormal{and}\  \alpha,\beta,\gamma\in(\mathbb{R}^3)^*.$$
Consider the vector space $$\mathfrak{g}:=\mathfrak{sl}(3)\oplus \mathbb{R}^3\oplus \left(\mathbb{R}^3\right)^*$$ endowed with the Lie bracket $\left[\cdot,\cdot\right]$  defined by the following relations:\\

\begin{itemize}
	\item[$i)$] The Lie bracket of two matrices $X,Y\in\mathfrak{sl}(3)$ is given by $$[X,Y]:=XY-YX\in\mathfrak{sl}(3).$$
	\item[$ii)$] If $X\in\mathfrak{sl}(3)$ and $v\in\mathbb{R}^3$ then 
	$$[X,v]:=Xv\in\mathbb{R}^3.$$
	\item[$iii)$] If $X\in\mathfrak{sl}(3)$ and $\alpha\in\left(\mathbb{R}^3\right)^*$ then
	$$[X,\alpha]:=-\alpha\circ X\in\left(\mathbb{R}^3\right)^*.$$
	\item[$iv)$] If $u,v\in\mathbb{R}^3$ then $$[u,v]:=-\frac{4}{3}T(u\wedge v)\in\left(\mathbb{R}^3\right)^*.$$
	\item[$v)$]] If $\alpha,\beta\in\left(\mathbb{R}^3\right)^*$ then
	$$[\alpha,\beta]:=\frac{4}{3}S(\alpha\wedge\beta)\in\mathbb{R}^3.$$
	\item[$vi)$] If $v=\sum\limits_{i=1}^3v^ie_i\in\mathbb{R}^3$ and $\alpha=\sum\limits_{j=1}^3\alpha^j\epsilon_j\in\left(\mathbb{R}^3\right)^*$ then
	$$[v,\alpha]:=(v^i\alpha^j)_{3\times 3}-\frac{1}{3}\alpha(v)\textnormal{I}_3\in\mathfrak{sl}(3),$$
	where $\textnormal{I}_3$ is the identity matrix of order 3.\\
 
	\item[$vii)$] For generic elements $X=X_1+u_1+\alpha_1,\ Y=X_2+u_2+\alpha_2\in\mathfrak{g},$ $X_j\in\mathfrak{sl}(3),\ u_j\in\mathbb{R}^3,\ \alpha_j\in\left(\mathbb{R}^3\right)^*,\ j=1,2;$ define the Lie bracket $[X,Y]$ by extending the rules $i)- vi)$ so that $[\cdot,\cdot]$ is bilinear and skew-symmetric.
\end{itemize}

 \
 
 The pair $\left(\mathfrak{g},\left[\cdot,\cdot\right]\right)$ is a non-compact Lie algebra which is isomorphic to the split real form of the complex simple Lie exceptional Lie algebra $\mathfrak{g}_2.$ For simplicity, we shall also denote this Lie algebra as $\mathfrak{g}_2.$\\
 
The set $\mathfrak{h},$ consisting of all diagonal and traceless $3\times 3$ matrices,  is a Cartan subalgebra $\mathfrak{g}_2.$ The corresponding root system is given by $$\Pi=\left\{\lambda_i-\lambda_j:1\leq i\neq j\leq 3\right\}\cup\left\{\pm\lambda_i:1\leq i\leq 3\right\},$$
where each $\lambda_i$ is defined as $$\begin{array}{rccc}
\lambda_i:&\mathfrak{h} &\longrightarrow & \mathbb{R}\\
&\textnormal{diag}(a_1,a_2,a_3)&\longmapsto&a_i
\end{array},\ i=1,2,3.$$
A set of positive roots can be chosen as $$\Pi^+=\left\{\lambda_i-\lambda_j:1\leq i<j\leq 3\right\}\cup\left\{\lambda_1,\lambda_2,-\lambda_3\right\}.$$ The corresponding set of simple roots is $\Sigma=\{\alpha_1:=\lambda_1-\lambda_2,\alpha_2:=\lambda_2\}.$ Given $i,j\in\{1,2,3\}$, let $E_{ij}$ be the $3\times 3$ square matrix whose $(i, j)$-entry is equal to 1, and all the other entries are zero. Then, the root spaces associated with $\Pi$ are
\begin{align*}
&\left(\mathfrak{g}_2\right)_{\lambda_i-\lambda_j}=\textnormal{span}\left\{E_{ij}\right\},\ 1\leq i\neq j\leq 3\\
&\left(\mathfrak{g}_2\right)_{\lambda_i}=\textnormal{span}\left\{e_i\right\},\textnormal{ and}\\
&\left(\mathfrak{g}_2\right)_{-\lambda_i}=\textnormal{span}\left\{\epsilon_i\right\},\ i=1,2,3.
\end{align*}

\subsection{Real flag manifolds of $\mathfrak{g}_2$}\label{section:2.3} A real flag manifold of the Lie algebra $\mathfrak{g}_2$ is defined as the homogeneous space $\mathbb{F} = G_2/P$, where $G_2$ is a connected non-compact simple Lie group with Lie algebra $\mathfrak{g}_2$, and $P$ is a parabolic subgroup of $G_2.$ The non-trivial parabolic subgroups of $G_2$ correspond bijectively to proper subsets $\Theta$ of the set $\Sigma$ of simple roots associated with the Cartan subalgebra $\mathfrak{h}$. This correspondence is given as follows: for a given $\Theta \subsetneq \Sigma$, we denote by $\langle\Theta\rangle$ the set of all roots that are linear combinations with integer coefficients of elements in $\Theta$. Let $\langle\Theta\rangle^+ = \langle\Theta\rangle \cap \Pi^+$ and $\langle\Theta\rangle^- = \langle\Theta\rangle \cap (-\Pi^+)$. We then define 
$$\mathfrak{p}_\Theta:=\mathfrak{h}\oplus\bigoplus\limits_{\alpha\in\Pi}(\mathfrak{g}_2)_{\alpha}\oplus\bigoplus\limits_{\alpha\in\langle\Theta\rangle^-}(\mathfrak{g}_2)_\alpha,$$
and $$P_\Theta=\{a\in G_2:\textnormal{Ad}(a)\mathfrak{p}_\Theta=\mathfrak{p}_\Theta\}.$$ Then $P_\Theta$ is a parabolic subgroup of $G_2$ and every non-trivial parabolic subgroup of $G_2$ is isomorphic to $P_\Theta$ for some $\Theta\subsetneq \Sigma.$ Consequently, there exist precisely three real flag manifolds of $\mathfrak{g}_2$, corresponding to $\Theta=\emptyset, \ \Theta=\{\alpha_1\},$ and $\Theta=\{\alpha_2\}.$\\

For each $\Theta \subsetneq \Sigma$, the maximal compact subgroup $$K\cong \text{SO}(4)\cong (\text{SU}(2) \times \text{SU}(2))/\{\pm(\text{I}, \text{I})\}$$ of $G_2$ acts transitively on the flag $\mathbb{F}_\Theta = G_2/P_\Theta$. The isotropy subgroup of $eP_\Theta$  is $K_\Theta:= K \cap P_\Theta$. Therefore, a flag manifold $\mathbb{F}_\Theta = G_2/P_\Theta$ can also be represented as $K/K_\Theta$. The Lie algebra 
 $\mathfrak{k}$ of $K$ is isomorphic to $\mathfrak{so}(4)=\mathfrak{so}(3)\oplus\mathfrak{so}(3)$ (see, for instance, \cite{S}). It is spanned by the vectors
\begin{equation*}
X_1:=E_{21}-E_{12},\ X_2:=E_{31}-E_{13},\ X_3:=E_{32}-E_{23},\ Y_i:=e_i-\epsilon_i,\ i=1,2,3.
\end{equation*}
which satisfy the following relations:
\begin{equation}\label{brackets:0}
	\begin{array}{lll}
	[X_1,X_2]=X_3, & [X_2,X_3]=X_1, & [X_3,Y_3]=-Y_2,\\
	\\
	\textnormal{[}X_1,X_3\textnormal{]}=-X_2, & [X_2,Y_1]=Y_3, & [Y_1,Y_2]=X_1+\frac{4}{3}Y_3,\\
	\\
	\textnormal{[}X_1,Y_1\textnormal{]}=Y_2, & [X_2,Y_3]=-Y_1, & [Y_1,Y_3]=X_2-\frac{4}{3}Y_2,\\
	\\
	\textnormal{[}X_1,Y_2\textnormal{]}=-Y_1, & [X_3,Y_2]=Y_3, & [Y_2,Y_3]=X_3+\frac{4}{3}Y_1.\\
	\end{array}
\end{equation}

\

The Killing form $B$ of $\mathfrak{k}\cong\mathfrak{so}(3)\oplus\mathfrak{so}(3)$ is negative definite. Consequently, $(\cdot,\cdot) := -B$ defines an $\Ad(K)$-invariant inner product on $\mathfrak{k}.$ Due to the compactness of $K$, we have that for each $\Theta \subsetneq \Sigma$, the flag manifold $\mathbb{F}_\Theta = K/K_\Theta$ is reductive. This implies the existence of a decomposition $\mathfrak{k} = \mathfrak{k}_\Theta\oplus \mathfrak{m}_\Theta,$ where $\mathfrak{k}_\Theta$ is the Lie algebra of $K_\Theta$ and $\Ad(k)\mathfrak{m}_\Theta = \mathfrak{m}_\Theta$ for all $k \in K_\Theta$. The isotropy representation of $\mathbb{F}_\Theta$ is equivalent to the representation
\begin{equation}\label{isotropy:representation}\textnormal{Ad}^{K_\Theta}\big{|}_{\mathfrak{m}_\Theta}:K_\Theta\longrightarrow \textnormal{GL}(\mathfrak{m}_\Theta),\end{equation}
which is completely reducible, that is, $\mathfrak{m}_\Theta$ can be decomposed as a direct sum of irreducible (possibly equivalent) subrepresentations. The following proposition proved by Patr\~ao and San Martin in \cite{PSM} gives the description of these subrepresentations and their equivalences.
\begin{proposition}\label{isotropy:representation:g_2} Let $\mathbb{F}_\Theta$ be a real flag manifold of $\mathfrak{g}_2.$ Then the following statements hold:
	\begin{itemize}
		\item[$a)$] If $\Theta=\emptyset$, then $\mathfrak{k}_\emptyset=\{0\}$ and $\mathfrak{m}_{\emptyset}$ is the direct sum of six one-dimensional $K_\emptyset$-submodules given by $$\textnormal{span}\left\{X_i\right\},\textnormal{span}\left\{Y_i\right\},\ i=1,2,3.$$ All these submodules are irreducible and the equivalence classes among them are \begin{align*}
		&\{\textnormal{span}\left\{X_1\right\},\textnormal{span}\left\{Y_3\right\}\},\\ &\{\textnormal{span}\left\{X_2\right\},\textnormal{span}\left\{Y_2\right\}\},\\ &\{\textnormal{span}\left\{X_3\right\},\textnormal{span}\left\{Y_1\right\}\}.
		\end{align*}
		\item[$b)$] If $\Theta=\{\alpha_1\}$, then $\mathfrak{k}_{\{\alpha_1\}}=\Span\{X_1\}$ and $\mathfrak{m}_{\{\alpha_1\}}$ decomposes into the irreducible $K_{\{\alpha_1\}}$-submodules 
		\begin{align*}
		&\textnormal{span}\left\{Y_3\right\},\\ &\textnormal{span}\left\{X_2,X_3\right\},\\ &\textnormal{span}\left\{Y_2,Y_1\right\},
		\end{align*}
		where the two-dimensional submodules are equivalent and the map $$\begin{array}{rccc}
  T:&\Span\{X_2,X_3\}&\longrightarrow&\Span\{Y_2,Y_1\}\\
  & X_2 & \longmapsto & Y_2\\
  & X_3 & \longmapsto & -Y_1
  \end{array}$$ is an equivariant isomorphism. 
		\item[$c)$] If $\Theta=\{\alpha_2\}$, then $\mathfrak{k}_{\{\alpha_2\}}=\Span\{Y_2\}$ and $\mathfrak{m}_{\{\alpha_2\}}$ decomposes into the irreducible inequivalent $K_{\{\alpha_2\}}$-submodules
		\begin{align*}
		&\textnormal{span}\left\{X_2\right\},\\ &\textnormal{span}\left\{(\sqrt{13}-2)X_3+3Y_1,(\sqrt{13}-2)X_1+3Y_3\right\},\\
		&\textnormal{span}\left\{(\sqrt{13}+2)X_3-3Y_1,(\sqrt{13}+2)X_1-3Y_3\right\}.
		\end{align*} 
		\end{itemize}
\end{proposition}
\begin{remark} The irreducible submodules provided in Proposition \ref{isotropy:representation:g_2} are not unique and are not pairwise orthogonal with respect to the inner product $(\cdot,\cdot)$. This lack of orthogonality complicates the characterization of invariant metrics. Therefore, we will consider different irreducible submodules of the isotropy representation \eqref{isotropy:representation} beyond those presented in Proposition \ref{isotropy:representation:g_2} (see Proposition \ref{new:submodules}).
\end{remark}
\begin{remark}\label{remark:2.3} The representation $$\Ad^{K_{\{\alpha_1\}}}\big{|}_{\Span\{X_2,X_3\}}:K_{\{\alpha_1\}}\rightarrow \textnormal{GL}(\Span\{X_2,X_3\})$$ is orthogonal. This means that the vector space $\textnormal{End}(\Span\{X_2,X_3\})$  of all equivariant endomorphism of $\Span\{X_2,X_3\}$ is isomorphic to $\mathbb{R}.$ As a consequence, any equivariant isomorphism from $\Span\{X_2,X_3\}$ to $\Span\{Y_2,Y_1\}$ is a scalar multiple of $T.$ Further details can be found in \cite{PSM} for reference.
\end{remark}

\begin{proposition} The vectors
\begin{align*}
    &W_i:=\frac{1}{2}X_i,\ i=1,2,3,\\
    &Z_1:=\frac{3}{2\sqrt{13}}Y_1-\frac{1}{\sqrt{13}}X_3, \hspace{0.4cm}
     Z_2:=\frac{3}{2\sqrt{13}}Y_2+\frac{1}{\sqrt{13}}X_2,\hspace{0.4cm}
     Z_3:=\frac{3}{2\sqrt{13}}Y_3-\frac{1}{\sqrt{13}}X_1,
\end{align*}
form an $(\cdot,\cdot)-$orthonormal basis for $\mathfrak{k},$ and satisfy the following bracket relations:
\begin{equation}\label{brackets:1}
    \begin{array}{lll}
	[W_1,W_2]=\frac{1}{2}W_3, & [W_2,W_3]=\frac{1}{2}W_1, & [W_3,Z_3]=-\frac{1}{2}Z_2,\\
	\\
	\textnormal{[}W_1,W_3\textnormal{]}=-\frac{1}{2}W_2, & [W_2,Z_1]=\frac{1}{2}Z_3, & [Z_1,Z_2]=\frac{1}{2}W_1,\\
	\\
	\textnormal{[}W_1,Z_1\textnormal{]}=\frac{1}{2}Z_2, & [W_2,Z_3]=-\frac{1}{2}Z_1, & [Z_1,Z_3]=\frac{1}{2}W_2,\\
	\\
	\textnormal{[}W_1,Z_2\textnormal{]}=-\frac{1}{2}Z_1, & [W_3,Z_2]=\frac{1}{2}Z_3, & [Z_2,Z_3]=\frac{1}{2}W_3.\\
	\end{array}
\end{equation}
\end{proposition}
\begin{proof} A straightforward calculation reveals that the non-zero inner products among the vectors $X_i$ and $Y_i$ (where $i = 1, 2, 3$) are
\begin{align*}
&(X_i, X_i) = 4, \quad (Y_i, Y_i) = \frac{68}{9}, \quad i = 1, 2, 3, \\
&(X_1, Y_3) = -(X_2, Y_2) = (X_3, Y_1) = \frac{8}{3}.
\end{align*}

The $(\cdot, \cdot)$-orthonormal basis $\{W_1, W_2, W_3, Z_1, Z_2, Z_3\}$ is obtained by applying the Gram-Schmidt algorithm to orthonormalize the basis $\{X_1, X_2, X_3, Y_1, Y_2, Y_3\}$ with respect to the inner product $(\cdot, \cdot).$ The relations \eqref{brackets:1} follow from a lengthy calculation using \eqref{brackets:0}.
\end{proof}
\begin{lemma}\label{lemma} Let $\rho: G \rightarrow \text{GL}(V)$ and $\tau: G \rightarrow \text{GL}(W)$ be equivalent representations of a Lie group $G$. Let $T: V \rightarrow W$ be an equivariant isomorphism. Define the graph of $T$ as the set $$\textnormal{graph}(T):=\{X+T(X)\in V\oplus W:X\in V\}\subseteq V\oplus W.$$ Then $\textnormal{graph}(T)$ is a vector space, and the map $$\gamma:G\rightarrow\textnormal{GL}(\textnormal{graph}(T)),\ \gamma(a)(X+T(X)):=\rho(a)X+\tau(a)T(X),\ a\in G,\ X\in V,$$ is a representation of $G$ that is equivalent to $\rho$.
\end{lemma}
\begin{proof} Since $T$ is linear, $\text{graph}(T)$ is a vector subspace of $V \oplus W$. The fact that $\gamma$ defines a representation of $G$ comes from the fact that $\rho$ and $\tau$ are representations of $G.$ Finally, the map $\tilde{T}: V \rightarrow \text{graph}(T)$, defined as $X \mapsto X + T(X)$, is trivially a linear isomorphism, and for every $a \in G$ and $X \in V$, we have
\begin{align*}
    \tilde{T}\rho(a)(X) &= \rho(a)X+T(\rho(a)X)\\
                        &= \rho(a)X+\tau(a)T(X), \hspace{1cm} \textnormal{since}\ T\ \textnormal{is equivariant},\\
                        &= \gamma(a)\left(X+T(X)\right)\\
                        &= \gamma(a)\tilde{T}(X).
\end{align*}
    That is, $\tilde{T}\rho(a)=\gamma(a)\Tilde{T},$ for all $a\in G.$ Hence $\tilde{T}$ is an equivariant isomorphism and, consequently, $\gamma$ and $\rho$ are equivalent.
\end{proof}
\begin{proposition}\label{new:submodules}
Let $\mathbb{F}_\Theta$ be a real flag manifold of $\mathfrak{g}_2.$ Then the following statements hold:
	\begin{itemize}
		\item[$a)$] If $\Theta=\emptyset$, the representation \eqref{isotropy:representation} can be decomposed into six one-dimensional submodules given by $$\textnormal{span}\left\{W_i\right\},\textnormal{span}\left\{Z_i\right\},\ i=1,2,3.$$ All these submodules are irreducible and the equivalence classes among them are \begin{align*}
		&\{\textnormal{span}\left\{W_1\right\},\textnormal{span}\left\{Z_3\right\}\},\\ &\{\textnormal{span}\left\{W_2\right\},\textnormal{span}\left\{Z_2\right\}\},\\ &\{\textnormal{span}\left\{W_3\right\},\textnormal{span}\left\{Z_1\right\}\}.
		\end{align*}
		\item[$b)$] If $\Theta=\{\alpha_1\}$, the representation \eqref{isotropy:representation} can be decomposed into the irreducible subrepresentations 
		\begin{align*}
		&\textnormal{span}\left\{Z_3\right\},\\ &\textnormal{span}\left\{W_2,W_3\right\},\\ &\textnormal{span}\left\{Z_2,-Z_1\right\},
		\end{align*}
		where the two-dimensional submodules are equivalent and the map $$\begin{array}{rccc}
  \tilde{T}:&\Span\{W_2,W_3\}&\longrightarrow&\Span\{Z_2,-Z_1\}\\
  & W_2 & \longmapsto & Z_2\\
  & W_3 & \longmapsto & -Z_1
  \end{array}$$ is an equivariant isomorphism. 
		\item[$c)$] If $\Theta=\{\alpha_2\}$, the representation \eqref{isotropy:representation} can be decomposed into the irreducible inequivalent subrepresentations
		\begin{align*}
		&\textnormal{span}\left\{\frac{\sqrt{13}W_2+2Z_2}{\sqrt{17}}\right\},\\ &\textnormal{span}\left\{\frac{W_1+Z_3}{\sqrt{2}},\frac{W_3+Z_1}{\sqrt{2}}\right\},\\
		&\textnormal{span}\left\{\frac{W_1-Z_3}{\sqrt{2}},\frac{W_3-Z_1}{\sqrt{2}}\right\}.
		\end{align*} 
		\end{itemize}
\end{proposition}
\begin{proof} For the proof of $a),$ let us assume that $\Theta=\emptyset.$ Then, by Proposition \ref{isotropy:representation:g_2}, the one-dimensional subspaces $\Span\{X_i\},\Span\{Y_i\},\ i=1,2,3$  are $K_\emptyset$-invariant irreducible subspaces of $\mathfrak{m}_\emptyset.$ Additionally, $\Span\{X_1\}$ is equivalent to $\Span\{Y_3\},$ $\Span\{X_2\}$ is equivalent to $\Span\{Y_2\}$ and $\Span\{X_3\}$ is equivalent to $\Span\{Y_1\}.$ Since these all submodules are one-dimensional, the linear maps \begin{align*}
    &\begin{array}{rccc}
         T_1:&\Span\{X_1\} &\longrightarrow & \Span\{Y_3\}\\
         & X_1  & \longmapsto & -\frac{3}{2}Y_3 
    \end{array},\\
    \\
    &\begin{array}{rccc}
         T_2:&\Span\{X_2\} &\longrightarrow & \Span\{Y_2\}\\
         & X_2  & \longmapsto & \frac{3}{2}Y_2 
    \end{array},\\
    \\
    &\begin{array}{rccc}
         T_3:&\Span\{X_3\} &\longrightarrow & \Span\{Y_1\}\\
         & X_3  & \longmapsto & -\frac{3}{2}Y_1 
    \end{array}
     \end{align*}
are equivariant isomorphisms. By Lemma \ref{lemma}, $\{\lambda X_1+T_1(\lambda X_1):\lambda\in\mathbb{R}\}$ is a $K_\emptyset$-invariant subspace of $\mathfrak{m}_\emptyset$ that is equivalent to $\Span\{X_1\},$ $\{\lambda X_2+T_1(\lambda X_2):\lambda\in\mathbb{R}\}$ is a $K_\emptyset$-invariant subspace of $\mathfrak{m}_\emptyset$ that is equivalent to $\Span\{X_2\},$ and $\{\lambda X_3+T_1(\lambda X_3):\lambda\in\mathbb{R}\}$ is a $K_\emptyset$-invariant subspace of $\mathfrak{m}_\emptyset$ that is equivalent to $\Span\{X_3\}.$ The statement $a)$ now follows from the fact that
\begin{align*}
    \{\lambda X_1+T_1(\lambda X_1):\lambda\in\mathbb{R}\} &=\left\{\lambda X_1-\lambda\frac{3}{2}Y_3:\lambda\in\mathbb{R}\right\}\\
    &=\left\{\lambda\left( X_1-\frac{3}{2}Y_3\right):\lambda\in\mathbb{R}\right\}\\
    &=\Span\left\{X_1-\frac{3}{2}Y_3\right\}\\
    &=\Span\left\{-\sqrt{13}Z_3\right\}\\
    &=\Span\{Z_3\},
\end{align*}
\begin{align*}
    \{\lambda X_2+T_2(\lambda X_2):\lambda\in\mathbb{R}\} &=\left\{\lambda X_2+\lambda\frac{3}{2}Y_2:\lambda\in\mathbb{R}\right\}\\
    &=\left\{\lambda\left( X_2+\frac{3}{2}Y_2\right):\lambda\in\mathbb{R}\right\}\\
    &=\Span\left\{X_2+\frac{3}{2}Y_2\right\}\\
    &=\Span\left\{\sqrt{13}Z_2\right\}\\
    &=\Span\{Z_2\},
\end{align*}
\begin{align*}
    \{\lambda X_3+T_3(\lambda X_3):\lambda\in\mathbb{R}\} &=\left\{\lambda X_3-\lambda\frac{3}{2}Y_1:\lambda\in\mathbb{R}\right\}\\
    &=\left\{\lambda\left( X_3-\frac{3}{2}Y_1\right):\lambda\in\mathbb{R}\right\}\\
    &=\Span\left\{X_3-\frac{3}{2}Y_1\right\}\\
    &=\Span\left\{-\sqrt{13}Z_1\right\}\\
    &=\Span\{Z_1\},
\end{align*}
and $\Span\{X_i\}=\Span\left\{\frac{3}{2}X_i\right\}=\Span\{W_i\},$ for $i=1,2,3.$\\

Let us prove $b).$ Assume that $\Theta=\{\alpha_1\}.$ By Proposition \ref{isotropy:representation:g_2}, the sets $$\Span\{Y_3\},\ \Span\{X_2,X_3\},\ \Span\{Y_2,Y_1\}$$ are $K_{\{\alpha_1\}}$-invariant irreducible subspaces of $\mathfrak{m}_{\{\alpha_1\}},$ and the linear map $$\begin{array}{rccc}
  T:&\Span\{X_2,X_3\}&\longrightarrow&\Span\{Y_2,Y_1\}\\
  & X_2 & \longmapsto & Y_2\\
  & X_3 & \longmapsto & -Y_1
  \end{array}$$ is an equivariant isomorphism. Then, $\frac{3}{2}T$ is also an equivariant isomorphism. By Lemma \ref{lemma}, the set $$\left\{\lambda_1X_2+\lambda_2 X_3+\frac{3}{2}T\left(\lambda_1X_2+\lambda_2 X_3\right):\lambda_1,\lambda_2\in\mathbb{R}\right\}$$ is a $K_{\{\alpha_1\}}$-invariant irreducible subspace of $\mathfrak{m}_{\{\alpha_1\}}$ that is equivalent to $\Span\{X_2,X_3\}.$ But
  
  \begin{align*}
      &\left\{\lambda_1X_2+\lambda_2 X_3+\frac{3}{2}T\left(\lambda_1X_2+\lambda_2 X_3\right):\lambda_1,\lambda_2\in\mathbb{R}\right\}\\
      \\
      =\ &\left\{\lambda_1X_2+\lambda_2 X_3+\lambda_1\frac{3}{2}Y_2-\lambda_2 \frac{3}{2}Y_1:\lambda_1,\lambda_2\in\mathbb{R}\right\}\\
      \\
      =\ &\left\{\lambda_1\left(X_2+\frac{3}{2}Y_2\right)+\lambda_2\left( X_3-\frac{3}{2}Y_1\right):\lambda_1,\lambda_2\in\mathbb{R}\right\}\\
      \\
      =\ &\left\{\lambda_1\left(\sqrt{13}Z_2\right)+\lambda_2\left(-\sqrt{13}Z_1\right):\lambda_1,\lambda_2\in\mathbb{R}\right\}\\
      \\
      =\ &\Span\{Z_2,-Z_1\}.
  \end{align*}
  Since $\Span\{X_2,X_3\}=\Span\left\{\frac{1}{2}X_2,\frac{1}{2}X_3\right\}=\Span\{W_2,W_3\},$ then $\Span\{W_2,W_3\}$ and $\Span\{Z_2,-Z_1\}$ are $K_{\{\alpha_1\}}$-invariant irreducible equivalent subspaces of $\mathfrak{m}_{\{\alpha_1\}}.$ Moreover, from the proof of Lemma \ref{lemma}, we have that the map $$\Span\{W_2,W_3\}\ni X\mapsto X+\frac{3}{2}T(X)\in \Span\{Z_2,-Z_1\}$$ is an intertwining isomorphism. Therefore, the map $$\Span\{W_2,W_3\}\ni X\mapsto \frac{2}{\sqrt{13}}\left(X+\frac{3}{2}T(X)\right)\in \Span\{Z_2,-Z_1\}$$ is also an interwining isomorphism. This last map is nothing but $\tilde{T}$ since $$\frac{2}{\sqrt{13}}\left(W_2+\frac{3}{2}T(W_2)\right)=Z_2\ \textnormal{and}\ \frac{2}{\sqrt{13}}\left(W_3+\frac{3}{2}T(W_3)\right)=-Z_1.$$ So far we have shown that the subspaces $\Span\{W_2,W_3\}$ and $\Span\{Z_2,-Z_1\}$ are $K_{\{\alpha_1\}}$-invariant irreducible and equivalent, and that $\tilde{T}$ is an intertwining isomorphism between them. To prove that $\Span\{Z_3\}$ is $K_{\{\alpha_1\}}$-invariant, consider the representation $$\Ad^{K_{\{\alpha_1\}}}\big{|}_{\mathfrak{k}_{\{\alpha_1\}}\oplus\Span\{Y_3\}}:K_{\{\alpha_1\}}\rightarrow\textnormal{GL}\left(\mathfrak{k}_{\{\alpha_1\}}\oplus\Span\{Y_3\}\right).$$ The inner product $(\cdot,\cdot)$ is $\Ad(K_{\{\alpha_1\}})$-invariant because it is $\Ad(K)$-invariant. Additionally, $(X_1,Z_1)=0,$ so the vector space $\Span\{Z_3\}$ is the orthogonal complement of $\mathfrak{k}_{\{\alpha_1\}}=\Span\{X_1\}$ in $\mathfrak{k}_{\{\alpha_1\}}\oplus\Span\{Y_3\}$ with respect to $(\cdot,\cdot).$ Since $\mathfrak{k}_{\{\alpha_1\}}$ is the Lie algebra of $K_{\{\alpha_1\}},$ then $\mathfrak{k}_{\{\alpha_1\}}$ is $K_{\{\alpha_1\}}$-invariant, so is its $(\cdot,\cdot)$-orthogonal complement $\Span\{Z_3\}.$ This completes the proof of $b).$\\

  To prove $c),$ let us consider the case where $\Theta=\{\alpha_2\}.$  Arguing as before, observe that $(\sqrt{13}W_2+2Z_2,Y_2)=0,$ so $\Span\{\sqrt{13}W_2+2Z_2\}$ is the $(\cdot,\cdot)$-orthogonal complement of $\mathfrak{k}_{\{\alpha_2\}}$ in $\mathfrak{k}_{\{\alpha_2\}}\oplus\Span\{\sqrt{13}W_2+2Z_2\}.$ Since $\mathfrak{k}_{\{\alpha_2\}}$ is $K_{\{\alpha_2\}},$ and $(\cdot,\cdot)$ is $\Ad(K_{\{\alpha_2\}})$-invariant, then $\Span\{\sqrt{13}W_2+2Z_2\}$ is $K_{\{\alpha_2\}}$-invariant. On the other hand, observe that 
  \begin{align*}
      &\frac{W_3+Z_1}{\sqrt{2}}=\frac{1}{2\sqrt{26}}\left((\sqrt{13}-2)X_3+3Y_1\right),\ \frac{W_1+Z_3}{\sqrt{2}}=\frac{1}{2\sqrt{26}}\left((\sqrt{13}-2)X_1+3Y_3\right)\\
      \\
      &\frac{W_3-Z_1}{\sqrt{2}}=\frac{1}{2\sqrt{26}}\left((\sqrt{13}+2)X_3-3Y_1\right),\ \frac{W_1+Z_3}{\sqrt{2}}=\frac{1}{2\sqrt{26}}\left((\sqrt{13}+2)X_1-3Y_3\right).
  \end{align*}
  Hence
  \begin{align*}
      &\textnormal{span}\left\{(\sqrt{13}-2)X_3+3Y_1,(\sqrt{13}-2)X_1+3Y_3\right\}=\textnormal{span}\left\{\frac{W_3+Z_1}{\sqrt{2}},\frac{W_1+Z_3}{\sqrt{2}}\right\}
  \end{align*}
  and
  \begin{align*}
      &\textnormal{span}\left\{(\sqrt{13}+2)X_3-3Y_1,(\sqrt{13}+2)X_1-3Y_3\right\}=\textnormal{span}\left\{\frac{W_3-Z_1}{\sqrt{2}},\frac{W_1-Z_3}{\sqrt{2}}\right\}.
  \end{align*}
  The proof is complete.
  \end{proof}
\section{Homogeneous geodesics on flag manifolds}\label{sec:hom_geo}
\subsection{$K$-invariant metrics} In this section we present a useful description of the $K$-invariant metrics on the real flag manifolds of $\mathfrak{g}_2.$ As in Section \ref{section:2.3}, we will fix the $\Ad(K)$-invariant inner product $(\cdot,\cdot)$ defined as the negative of the Killing form on $\mathfrak{k}.$  

\begin{theorem}\label{metrics:theorem} Let $\mathbb{F}_\Theta=K/K_\Theta$ be a flag of $\mathfrak{g}_2.$
	\begin{itemize}
		\item[$a)$] \allowbreak If $\Theta=\emptyset$, then for each metric operator $A:\mathfrak{m}_\emptyset\rightarrow\mathfrak{m}_\emptyset$ associated with an $\Ad(K_\emptyset)$-invariant inner product on $\mathfrak{m}_\emptyset,$ there exist positive numbers $\mu_1,...,\mu_6,$ and real numbers $a_1\in(-\sqrt{\mu_1\mu_{2}},\sqrt{\mu_1\mu_{2}}),$ $a_2\in(-\sqrt{\mu_3\mu_4},\sqrt{\mu_3\mu_4}),$ and $a_3\in(-\sqrt{\mu_5\mu_6},\sqrt{\mu_5\mu_6})$ such that $A$ is written the $(\cdot,\cdot)$-orthonormal basis $\mathcal{B}_\emptyset=\{W_1,Z_3,W_2,Z_2,W_3,Z_1\}$ as 
		\begin{equation}\label{metric:emptyset}
  \left[A\right]_{\mathcal{B}_\emptyset}=\left(\begin{array}{cccccc}\mu_1 & a_1 & 0 & 0 & 0 & 0\\
		a_1 & \mu_2 & 0 & 0 & 0 & 0\\
		0 & 0 & \mu_3 & a_2 & 0 & 0\\
		0 & 0 & a_2 & \mu_4 & 0 & 0\\
		0 & 0 & 0 & 0 & \mu_5 & a_3\\
		0 & 0 & 0 & 0 & a_3 & \mu_6
		\end{array}\right).
		\end{equation} This representation covers all metric operators associated with $\Ad(K_\emptyset)$-invariant inner products.\\
  
		\item[$b)$] If $\Theta=\{\alpha_1\}$, then for each metric operator $A:\mathfrak{m}_{\{\alpha_1\}}\rightarrow\mathfrak{m}_{\{\alpha_1\}}$ associated with an $\Ad(K_{\{\alpha_1\}})$-invariant inner product on $\mathfrak{m}_{\{\alpha_1\}},$ there exist positive numbers $\mu_1,\mu_2,\mu_3,$ and a real number $a\in(-\sqrt{\mu_2\mu_3},\sqrt{\mu_2\mu_3})$ such that $A$ is written in the basis $\mathcal{B}_{\{\alpha_1\}}=\{Z_3,W_2,W_3,Z_2,-Z_1\}$ as
		\begin{equation}\label{metric:alpha_1}
		\left[A\right]_{\mathcal{B}_{\{\alpha_1\}}}=\left(\begin{array}{ccccc}\mu_1 & 0 & 0 & 0 & 0\\
		0 & \mu_2 & 0 & a & 0 \\
		0 & 0 & \mu_2 & 0 & a \\
		0 & a & 0 & \mu_3 & 0 \\
		0 & 0 & a & 0 & \mu_3  \\
		\end{array}\right).
		\end{equation}
        This representation covers all metric operators associated with $\Ad\left(K_{\{\alpha_1\}}\right)$-invariant inner products.\\
        
		\item[$c)$] If $\Theta=\{\alpha_2\}$, then there exist positive real numbers $\mu_i,\ i=1,2,3$ such that the metric operator $A$ associated with an $\Ad\left(K_{\{\alpha_2\}}\right)$-invariant inner product on $\mathfrak{m}_{\{\alpha_2\}}$ is written in the basis $\mathcal{B}_{\{\alpha_2\}}=\left\{\frac{\sqrt{13}W_2+2Z_2}{\sqrt{17}},\frac{W_1+Z_3}{\sqrt{2}},\frac{W_3+Z_1}{\sqrt{2}},\frac{W_1-Z_3}{\sqrt{2}},\frac{W_3-Z_1}{\sqrt{2}}\right\}$ as
		\begin{equation}\label{metric:alpha_2}
		\left[A\right]_{\mathcal{B}_{\{\alpha_2\}}}=\left(\begin{array}{ccccc}\mu_1 & 0 & 0 & 0 & 0\\
		0 & \mu_2 & 0 & 0 & 0 \\
		0 & 0 & \mu_2 & 0 & 0 \\
		0 & 0 & 0 & \mu_3 & 0 \\
		0 & 0 & 0 & 0 & \mu_3  \\
		\end{array}\right).
		\end{equation}
        This representation covers all metric operators associated with $\Ad\left(K_{\{\alpha_2\}}\right)$-invariant inner products.\\
	\end{itemize}
\end{theorem}
\begin{proof} 
Assuming $\Theta=\emptyset,$ Proposition \ref{new:submodules} implies that $$\mathfrak{m}_\emptyset=\left(\mathfrak{m}_\emptyset\right)_1\oplus\left(\mathfrak{m}_\emptyset\right)_2\oplus\left(\mathfrak{m}_\emptyset\right)_3\oplus\left(\mathfrak{m}_\emptyset\right)_4\oplus\left(\mathfrak{m}_\emptyset\right)_5\oplus\left(\mathfrak{m}_\emptyset\right)_6,$$ where
\begin{align*}
    \left(\mathfrak{m}_\emptyset\right)_1=\Span\{W_1\},\hspace{1cm} &\left(\mathfrak{m}_\emptyset\right)_2=\Span\{Z_3\},\\
    \left(\mathfrak{m}_\emptyset\right)_3=\Span\{W_2\},\hspace{1cm} &\left(\mathfrak{m}_\emptyset\right)_4=\Span\{Z_2\},\\
    \left(\mathfrak{m}_\emptyset\right)_5=\Span\{W_3\},\hspace{1cm} &\left(\mathfrak{m}_\emptyset\right)_6=\Span\{Z_1\},
\end{align*}
with $\left(\mathfrak{m}_\emptyset\right)_1$ equivalent to $\left(\mathfrak{m}_\emptyset\right)_2,$ $\left(\mathfrak{m}_\emptyset\right)_3$ equivalent to $\left(\mathfrak{m}_\emptyset\right)_4,$ and $\left(\mathfrak{m}_\emptyset\right)_5$ equivalent to $\left(\mathfrak{m}_\emptyset\right)_6$ (as $K_\emptyset$-modules). Define the linear maps
\begin{align*}
    &\left(T_\emptyset\right)_1^2:\left(\mathfrak{m}_\emptyset\right)_1\ni W_1\longmapsto Z_3\in\left(\mathfrak{m}_\emptyset\right)_2,\hspace{0.5cm}\left(T_\emptyset\right)_2^1:\left(\mathfrak{m}_\emptyset\right)_2\ni Z_3\longmapsto W_1\in\left(\mathfrak{m}_\emptyset\right)_1,\\
    &\left(T_\emptyset\right)_3^4:\left(\mathfrak{m}_\emptyset\right)_3\ni W_2\longmapsto Z_2\in\left(\mathfrak{m}_\emptyset\right)_4,\hspace{0.5cm}\left(T_\emptyset\right)_4^3:\left(\mathfrak{m}_\emptyset\right)_4\ni Z_2\longmapsto W_2\in\left(\mathfrak{m}_\emptyset\right)_3,\\
    &\left(T_\emptyset\right)_5^6:\left(\mathfrak{m}_\emptyset\right)_5\ni W_3\longmapsto Z_1\in\left(\mathfrak{m}_\emptyset\right)_6,\hspace{0.5cm}\left(T_\emptyset\right)_6^5:\left(\mathfrak{m}_\emptyset\right)_6\ni Z_1\longmapsto W_3\in\left(\mathfrak{m}_\emptyset\right)_5,\\
    &\left(T_\emptyset\right)_i^i:=\textnormal{I}_{\left(\mathfrak{m}_\emptyset\right)_{i}},\ i=1,2,3,4,5,6,\\
\end{align*} and $\left(T_\emptyset\right)_i^j:=0$ for any other $i,j,$ and consider the $(\cdot,\cdot)$-orthonormal sets $$\mathcal{B}^{\emptyset}_1:=\{W_1\},\ \mathcal{B}^{\emptyset}_2:=\{Z_3\},\ \mathcal{B}^{\emptyset}_3:=\{W_2\},\ \mathcal{B}^{\emptyset}_4:=\{Z_2\},\ \mathcal{B}^{\emptyset}_5:=\{W_3\},\ \mathcal{B}^{\emptyset}_6:=\{Z_1\}.$$ Then, the family $\{\left(T_\emptyset\right)_i^j:i,j\in\{1,2,3,4,5,6\}\}$ satisfies the properties $i)-iv)$ listed in Section \ref{section:2.1}, and $\mathcal{B}_\emptyset=\mathcal{B}^{\emptyset}_1\cup\mathcal{B}^{\emptyset}_2\cup\mathcal{B}^{\emptyset}_3\cup\mathcal{B}^{\emptyset}_4\cup\mathcal{B}^{\emptyset}_5\cup\mathcal{B}^{\emptyset}_6$ is an $(\cdot,\cdot)$-orthonormal basis of $\mathfrak{m}_\emptyset$ such that $\left(T_\emptyset\right)_i^j(\mathcal{B}^{\emptyset}_i)=\mathcal{B}^{\emptyset}_j$ whenever $\left(\mathfrak{m}_\emptyset\right)_{i}$ is equivalent to $\left(\mathfrak{m}_\emptyset\right)_{j}.$ Consequently, by  formula \eqref{positive:matrix}, every metric operator $A:\mathfrak{m}_\emptyset\rightarrow\mathfrak{m}_\emptyset$ is determined by a positive numbers $\mu_1,...,\mu_6$ and real number $a_1,a_2,a_3$ through the relation \begin{equation*}
\left[A\right]_{\mathcal{B}_\emptyset}=\left(\begin{array}{cccccc}\mu_1 & a_1 & 0 & 0 & 0 & 0\\
		a_1 & \mu_2 & 0 & 0 & 0 & 0\\
		0 & 0 & \mu_3 & a_2 & 0 & 0\\
		0 & 0 & a_2 & \mu_4 & 0 & 0\\
		0 & 0 & 0 & 0 & \mu_5 & a_3\\
		0 & 0 & 0 & 0 & a_3 & \mu_6
		\end{array}\right),
		\end{equation*}
  where this matrix is positive definite. Observe that the eigenvalues of this matrix are \begin{align*}
  &\frac{\mu_1+\mu_2\pm\sqrt{(\mu_1-\mu_2)^2+4a_1^2}}{2},\\
  &\frac{\mu_3+\mu_4\pm\sqrt{(\mu_3-\mu_4)^2+4a_2^2}}{2},\\ &\frac{\mu_5+\mu_6\pm\sqrt{(\mu_5-\mu_6)^2+4a_3^2}}{2}.
  \end{align*} Thus, the matrix is positive definite if and only if \begin{align*}
  &\mu_1+\mu_2\pm\sqrt{(\mu_1-\mu_2)^2+4a_1^2}>0,\\
  &\mu_3+\mu_4\pm\sqrt{(\mu_3-\mu_4)^2+4a_2^2}>0,\ \text{and}\\ &\mu_5+\mu_6\pm\sqrt{(\mu_5-\mu_6)^2+4a_3^2}>0,
  \end{align*} which is equivalent to $$|a_1|<\sqrt{\mu_1\mu_2},\ |a_2|<\sqrt{\mu_3\mu_4},\ \textnormal{and}\ |a_3|<\sqrt{\mu_5\mu_6}.$$ This proves $a).$\\

  For the proof of $b),$ we can proceed analogously to the proof of $a):$ let $\Theta=\{\alpha_1\},$ and consider $$\left(\mathfrak{m}_{\{\alpha_1\}}\right)_1:=\Span\{Z_3\},\ \left(\mathfrak{m}_{\{\alpha_1\}}\right)_2:=\Span\{W_2,W_3\},\ \textnormal{and}\ \left(\mathfrak{m}_{\{\alpha_1\}}\right)_3:=\Span\{Z_2,-Z_1\}.$$ Then, by Proposition \ref{new:submodules}, we have that $$\mathfrak{m}_{\{\alpha_1\}}=\left(\mathfrak{m}_{\{\alpha_1\}}\right)_1\oplus\left(\mathfrak{m}_{\{\alpha_1\}}\right)_2\oplus\left(\mathfrak{m}_{\{\alpha_1\}}\right)_3,$$ where $\left(\mathfrak{m}_{\{\alpha_1\}}\right)_2$ is equivalent to $\left(\mathfrak{m}_{\{\alpha_1\}}\right)_3,$ and $$\begin{array}{rccc}
  \tilde{T}:&\left(\mathfrak{m}_{\{\alpha_1\}}\right)_2&\longrightarrow&\left(\mathfrak{m}_{\{\alpha_1\}}\right)_3\\
  & W_2 & \longmapsto & Z_2\\
  & W_3 & \longmapsto & -Z_1
  \end{array}$$ is an intertwining isomorphism. In this case, we define \begin{equation*}
    \left(T_{\{\alpha_1\}}\right)_2^3:=\tilde{T},\hspace{0.5cm}\left(T_{\{\alpha_1\}}\right)_3^2:=\tilde{T}^{-1},\ \left(T_{\{\alpha_1\}}\right)_i^i:=\textnormal{I}_{\left(\mathfrak{m}_{\{\alpha_1\}}\right)_{i}},\ i=1,2,3,
\end{equation*}
$\left(T_{\{\alpha_1\}}\right)_i^j:=0$ for any other $i,j,$ and $$\mathcal{B}^{\{\alpha_1\}}_1:=\{Z_3\},\ \mathcal{B}^{\{\alpha_1\}}_2:=\{W_2,W_3\},\ \mathcal{B}^{\{\alpha_1\}}_3:=\{Z_2,-Z_1\}.$$ Again, the family $\{\left(T_{\{\alpha_1\}}\right)_i^j:i,j\in\{1,2,3\}\}$ satisfies the conditions $i)-iv)$ of Section \ref{section:2.1}, and $$\mathcal{B}_{\{\alpha_1\}}=\mathcal{B}^{\{\alpha_1\}}_1\cup\mathcal{B}^{\{\alpha_1\}}_2\cup\mathcal{B}^{\{\alpha_1\}}_3$$ is an $(\cdot,\cdot)$-orthonormal basis of $\mathfrak{m}_{\{\alpha_1\}}$ such that is an $(\cdot,\cdot)$-orthonormal basis of $\mathfrak{m}_\emptyset$ such that $\left(T_\emptyset\right)_i^j\left(\mathcal{B}^{\{\alpha_1\}}_i\right)=\mathcal{B}^{\{\alpha_1\}}_j$ whenever $\left(\mathfrak{m}_\emptyset\right)_{i}$ is equivalent to $\left(\mathfrak{m}_\emptyset\right)_{j}.$  Thus, by formula \eqref{positive:matrix}, any metric operator $A:\mathfrak{m}_{\{\alpha_1\}}\rightarrow\mathfrak{m}_{\{\alpha_1\}}$ is determined by positive numbers $\mu_1,\mu_2,\mu_3$ and real numbers $a,b,c,d$ through the formula 
\begin{equation*}
		\left[A\right]_{\mathcal{B}_{\{\alpha_1\}}}=\left(\begin{array}{ccccc}\mu_1 & 0 & 0 & 0 & 0\\
		0 & \mu_2 & 0 & a & c \\
		0 & 0 & \mu_2 & b & d \\
		0 & a & b & \mu_3 & 0 \\
		0 & c & d & 0 & \mu_3  \\
		\end{array}\right),
		\end{equation*}
    where the matrix is positive definite. The submatrix $$\left(\begin{array}{cc}a & b \\ c & d \end{array}\right)$$ defines the map $$\begin{array}{rccc}
    A_{32}:&\left(\mathfrak{m}_{\{\alpha_1\}}\right)_2&\longrightarrow& \left(\mathfrak{m}_{\{\alpha_1\}}\right)_3\\
     & W_2 & \longmapsto & aZ_2-cZ_1\\
     & W_3 & \longmapsto & bZ_2-dZ_1,
     \end{array}$$ and it is an equivariant map. Hence, since $\left(\mathfrak{m}_{\{\alpha_1\}}\right)_2$ is an orthogonal $K_{\{\alpha_1\}}$-module (see Remark \ref{remark:2.3}), then $A_{32}$ must be a multiple of $\tilde{T}.$ This implies that there exists $\lambda\in \mathbb{R}$ such that $$A_{32}(W_2)=\lambda\tilde{T}(W_2),\ \textnormal{and}\ A_{32}(W_3)=\lambda\tilde{T}(W_3),$$ or, equivalently, $$aZ_2-cZ_1=\lambda Z_2,\ \textnormal{and}\ bZ_2-dZ_1=-\lambda Z_1,$$ that is, $a=d=\lambda,$ and $b=c=0.$ Therefore, 
     \begin{equation*}
		\left[A\right]_{\mathcal{B}_{\{\alpha_1\}}}=\left(\begin{array}{ccccc}\mu_1 & 0 & 0 & 0 & 0\\
		0 & \mu_2 & 0 & a & 0 \\
		0 & 0 & \mu_2 & 0 & a \\
		0 & a & 0 & \mu_3 & 0 \\
		0 & 0 & a & 0 & \mu_3  \\
		\end{array}\right).
	\end{equation*}
    By computing the eigenvalues of this matrix, we obtain
    $$\mu_1,\ \textnormal{and}\ \frac{\mu_2+\mu_3\pm\sqrt{(\mu_2-\mu_3)^2+4a^2}}{2}.$$ Thus, it is positive definite if and only if $|a|<\sqrt{\mu_2\mu_3}.$ We have proven $b).$\\

    For $\Theta=\{\alpha_2\}$, Proposition \ref{new:submodules} gives us the decomposition $$\mathfrak{m}_{\{\alpha_2\}}=\left(\mathfrak{m}_{\{\alpha_2\}}\right)_1\oplus\left(\mathfrak{m}_{\{\alpha_2\}}\right)_2\oplus\left(\mathfrak{m}_{\{\alpha_2\}}\right)_3,$$ where 
    \begin{align*}
    &\left(\mathfrak{m}_{\{\alpha_2\}}\right)_1=\Span\left\{\frac{\sqrt{13}W_2+2Z_2}{\sqrt{17}}\right\},\\
    \\
    &\left(\mathfrak{m}_{\{\alpha_2\}}\right)_2=\Span\left\{\frac{W_1+Z_3}{\sqrt{2}},\frac{W_3+Z_1}{\sqrt{2}}\right\}\\
    \\
    &\left(\mathfrak{m}_{\{\alpha_2\}}\right)_3=\Span\left\{\frac{W_1-Z_3}{\sqrt{2}},\frac{W_3-Z_1}{\sqrt{2}}\right\},
    \end{align*}
    and all of them are not equivalent. In this case, since there are not nonzero equivariant maps between these submodules, we have that any metric operator $A:\mathfrak{m}_{\{\alpha_2\}}\rightarrow\mathfrak{m}_{\{\alpha_2\}}$ is determined by positive numbers $\mu_1,\mu_2,\mu_3$ such that $$A\big{|}_{\left(\mathfrak{m}_{\{\alpha_2\}}\right)_i}=\mu_i\textnormal{I}_{\left(\mathfrak{m}_{\{\alpha_2\}}\right)_i},\ i=1,2,3.$$  The sets 
    \begin{align*}
        &\mathcal{B}^{\{\alpha_2\}}_1:=\left\{\frac{\sqrt{13}W_2+2Z_2}{\sqrt{17}}\right\},\\
        \\
        &\mathcal{B}^{\{\alpha_2\}}_2:=\left\{\frac{W_1+Z_3}{\sqrt{2}},\frac{W_3+Z_1}{\sqrt{2}}\right\},\\
        \\
        &\mathcal{B}^{\{\alpha_2\}}_3:=\left\{\frac{W_1-Z_3}{\sqrt{2}},\frac{W_3-Z_1}{\sqrt{2}}\right\}
        \end{align*} are $(\cdot,\cdot)$-orthonormal bases of $\left(\mathfrak{m}_{\{\alpha_2\}}\right)_1,\ \left(\mathfrak{m}_{\{\alpha_2\}}\right)_2,\ \left(\mathfrak{m}_{\{\alpha_2\}}\right)_3$ respectively. Thus, $A$ can be written in the basis $\mathcal{B}_{\{\alpha_2\}}=\mathcal{B}^{\{\alpha_2\}}_1\cup\mathcal{B}^{\{\alpha_2\}}_2\cup\mathcal{B}^{\{\alpha_2\}}_3$ as 
        \begin{equation*}
		\left[A\right]_{\mathcal{B}_{\{\alpha_2\}}}=\left(\begin{array}{ccccc}\mu_1 & 0 & 0 & 0 & 0\\
		0 & \mu_2 & 0 & 0 & 0 \\
		0 & 0 & \mu_2 & 0 & 0 \\
		0 & 0 & 0 & \mu_3 & 0 \\
		0 & 0 & 0 & 0 & \mu_3  \\
		\end{array}\right).
		\end{equation*}
    The proof is complete.
\end{proof}

\subsection{G.o. metrics} This section is dedicated to classify the $K$-invariant metrics on flag manifolds of $\mathfrak{g}_2$ that are g.o. metrics.

\begin{definition} Let $G$ be a compact Lie group, $H$ a closed subgroup of $G,$ and $g$ a $G$-invariant metric on the homogeneous space $G/H.$  A smooth curve $\gamma$ on $G/H$ is called {\it homogeneous} if it is the orbit of a one-parameter subgroup of $G,$ that is, there exists $X$ in the Lie algebra $\mathfrak{g}$ of $G$ such that $$\gamma(t)=\exp(tX)H,$$ for all $t$ in the domain of $\gamma.$ If, in addition, the homogeneous curve $\gamma$ is a geodesic on $(G/H,g),$ then we say that $\gamma$ is a {\it homogeneous geodesic} with respect to $g.$ In such a case, the vector $X\in\mathfrak{g}$ is called a {\it geodesic vector}.
\end{definition}

\begin{definition} A $G$-invariant metric on a homogeneous space $G/H$ is a {\it g.o. metric} all geodesics on $(G/H,g)$ starting at $eH$ are homogeneous geodesics. 
\end{definition}

With the notations of Section \ref{section:2.1}, we recall that the set of $G$-invariant metrics $g$ on a homogeneous space $G/H$ are in bijective correspondence with the set of $\Ad(H)$-invariant inner products $\langle\cdot,\cdot\rangle$ on $\mathfrak{m}.$ So, in what follows we will write $g$ or $\langle\cdot,\cdot\rangle$ interchangeably to refer to either a $G$-invariant metric on $G/H$ or an $\Ad(H)$-invariant inner product on $\mathfrak{m}.$ \\

The following Proposition was proved by Souris in \cite{N}. It provides a helpful tool to determine whether a $G$-invariant metric is a g.o. metric.

\begin{proposition}\label{Souris:1} Let $\langle\cdot,\cdot\rangle$ be a $\Ad(H)$-invariant inner product on a homogeneous space $G/H,$ and let $A$ be the metric operator associated with $\langle\cdot,\cdot\rangle.$ Then $\langle\cdot,\cdot\rangle$ is a g.o. metric if and only if for all $X\in\mathfrak{m},$ there exist a vector $Z\in\mathfrak{h}$ such that 
\begin{equation}\label{Souris:criteria}
[Z+X,AX]=0.
\end{equation} 
\end{proposition}
\begin{theorem}\label{g.o.metrics:theorem} Let $\mathbb{F}_\Theta$ be a flag of $\mathfrak{g}_2,$ $\langle\cdot,\cdot\rangle$ an $\Ad(K_\Theta)$-invariant inner product on $\mathfrak{m}_\Theta,$ and $A$ its associated metric operator. Let 
\begin{align*}
    &\mathcal{B}_\emptyset=\{W_1,Z_3,W_2,Z_2,W_3,Z_1\}\\
    &\mathcal{B}_{\{\alpha_1\}}=\{Z_3,W_2,W_3,Z_2,-Z_1\},\ \text{and}\\
    &\mathcal{B}_{\{\alpha_2\}}=\left\{\frac{\sqrt{13}W_2+2Z_2}{\sqrt{17}},\frac{W_1+Z_3}{\sqrt{2}},\frac{W_3+Z_1}{\sqrt{2}},\frac{W_1-Z_3}{\sqrt{2}},\frac{W_3-Z_1}{\sqrt{2}}\right\}.
\end{align*}
\begin{itemize}
    \item[$a)$] If $\Theta=\emptyset,$ then $\langle\cdot,\cdot\rangle$ is a g.o. metric if and only if there exist $\mu>0,$ and $a\in(-\mu,\mu)$ such that \begin{equation}\label{g.o.metric:emptyset}
\left[A\right]_{\mathcal{B}_\emptyset}=\left(\begin{array}{cccccc}\mu & a & 0 & 0 & 0 & 0\\
		a & \mu & 0 & 0 & 0 & 0\\
		0 & 0 & \mu & -a & 0 & 0\\
		0 & 0 & -a & \mu & 0 & 0\\
		0 & 0 & 0 & 0 & \mu & a\\
		0 & 0 & 0 & 0 & a & \mu
		\end{array}\right).
		\end{equation}
    \item[$b)$] If $\Theta=\{\alpha_1\},$ then $\langle\cdot,\cdot\rangle$ is a g.o. metric if and only if there exist $\mu>\mu_1>0$ such that \begin{equation}\label{g.o.metric:alpha_1}
\left[A\right]_{\mathcal{B}_{\{\alpha_1\}}}=\left(\begin{array}{ccccc}\mu_1 & 0 & 0 & 0 & 0 \\
		0 & \mu & 0 & a & 0 \\
		0 & 0 & \mu & 0 & a\\
		0 & a & 0 & \mu & 0 \\
		0 & 0 & a & 0 & \mu
		\end{array}\right),\ \text{where}\ a^2=\mu(\mu-\mu_1).
		\end{equation}
    \item[$c)$] If $\Theta=\{\alpha_2\},$ then $\langle\cdot,\cdot\rangle$ is a g.o. metric if and only if there exist $\mu_1,\mu_2,\mu_3>0$ such that
    \begin{equation*}
		\left[A\right]_{\mathcal{B}_{\{\alpha_2\}}}=\left(\begin{array}{ccccc}\mu_1 & 0 & 0 & 0 & 0\\
		0 & \mu_2 & 0 & 0 & 0 \\
		0 & 0 & \mu_2 & 0 & 0 \\
		0 & 0 & 0 & \mu_3 & 0 \\
		0 & 0 & 0 & 0 & \mu_3  \\
		\end{array}\right),
  \end{equation*}
  where $\mu_1=\frac{34\mu_2\mu_3}{(2+\sqrt{13})^2\mu_2+(2-\sqrt{13})^2\mu_3}.$
\end{itemize}
\end{theorem}
\begin{proof} The proof of $a)$ and $b)$ can be found in \cite[Proposition 4.6]{GGN}. For the proof of $c),$ let $\Theta=\{\alpha_2\}.$ Then $$\mathfrak{k}_{\{\alpha_2\}}=\Span\{Y_2\}=\Span\{\sqrt{13}Z_2-2W_2\}.$$ In this case, Proposition \ref{Souris:1} says that $A$ is a g.o. metric if and only if for each $X\in\mathfrak{m}_{\{\alpha_2\}},$ there exists $\lambda\in\mathbb{R}$ such that \begin{equation}\label{auxiliar:2}[\lambda\left(\sqrt{13}Z_2-2W_2\right)+X,AX].
\end{equation} If a metric $A$ defined as $$\left[A\right]_{\mathcal{B}_{\{\alpha_2\}}}=\left(\begin{array}{ccccc}\mu_1 & 0 & 0 & 0 & 0\\
		0 & \mu_2 & 0 & 0 & 0 \\
		0 & 0 & \mu_2 & 0 & 0 \\
		0 & 0 & 0 & \mu_3 & 0 \\
		0 & 0 & 0 & 0 & \mu_3  \\
		\end{array}\right),$$
is a g.o. metric, then for $X=\sqrt{13}W_2+2Z_2+W_1+Z_3+W_3-Z_1$ there exists $\lambda\in\mathbb{R}$ such that the equation \eqref{auxiliar:2} is satisfied. Observe that
\begin{align*}
    A(\sqrt{13}W_2+2Z_2+W_1+Z_3)=\sqrt{13}\mu_1W_2+2\mu_1Z_2+\mu_2(W_1+Z_3)+\mu_3(W_3-Z_1)
\end{align*}
so that
\begin{align*}
    0=&[\lambda\left(\sqrt{13}Z_2-2W_2\right)+X,AX]\\
    =&\lambda[\sqrt{13}Z_2-2W_2,AX]+[X,AX]\\
    =&\lambda\{\sqrt{13}\mu_2[Z_2,W_1]+\sqrt{13}\mu_2[Z_2,Z_3]+\sqrt{13}\mu_3[Z_2,W_3]-\sqrt{13}\mu_3[Z_2,Z_1]\\
    &-2\mu_2[W_2,W_1]-2\mu_2[W_2,Z_3]-2\mu_3[W_2,W_3]+2\mu_3[W_2,Z_1]\}+\sqrt{13}\mu_2[W_2,W_1]\\
    &+\sqrt{13}\mu_2[W_2,Z_3]+\sqrt{13}\mu_3[W_2,W_3]-\sqrt{13}\mu_3[W_2,Z_1]+2\mu_2[Z_2,W_1]\\
    &+2\mu_2[Z_2,Z_3]+2\mu_3[Z_2,W_3]-2\mu_3[Z_2,Z_1]+\sqrt{13}\mu_1[W_1,W_2]+2\mu_1[W_1,Z_2]\\
    &+\mu_3[W_1,W_3]-\mu_3[W_1,Z_1]+\sqrt{13}\mu_1[Z_3,W_2]+2\mu_1[Z_3,Z_2]+\mu_3[Z_3,W_3]\\
    &-\mu_3[Z_3,Z_1]+\sqrt{13}\mu_1[W_3,W_2]+2\mu_1[W_3,Z_2]+\mu_2[W_3,W_1]+\mu_2[W_3,Z_3]\\
    &-\sqrt{13}\mu_1[Z_1,W_2]-2\mu_1[Z_1,Z_2]-\mu_2[Z_1,W_1]-\mu_2[Z_1,Z_3]\\
    =&\frac{\lambda}{2}\{\sqrt{13}\mu_2Z_1+\sqrt{13}\mu_2W_3-\sqrt{13}\mu_3Z_3+\sqrt{13}\mu_3W_1+2\mu_2W_3+2\mu_2Z_1-2\mu_3W_1\\
    &+2\mu_3Z_3\}-\frac{\sqrt{13}\mu_2}{2}W_3-\frac{\sqrt{13}\mu_2}{2}Z_1+\frac{\sqrt{13}\mu_3}{2}W_1-\frac{\sqrt{13}\mu_3}{2}Z_3+\frac{2\mu_2}{2}Z_1+\frac{2\mu_2}{2}W_3\\
    &-\frac{2\mu_3}{2}Z_3+\frac{2\mu_3}{2}W_1+\frac{\sqrt{13}\mu_1}{2}W_3-\frac{2\mu_1}{2}Z_1-\frac{\mu_3}{2}W_2-\frac{\mu_3}{2}Z_2+\frac{\sqrt{13}\mu_1}{2}Z_1-\frac{2\mu_1}{2}W_3\\
    &+\frac{\mu_3}{2}Z_2+\frac{\mu_3}{2}W_2-\frac{\sqrt{13}\mu_1}{2}W_1+\frac{2\mu_1}{2}Z_3+\frac{\mu_2}{2}W_2-\frac{\mu_2}{2}Z_2+\frac{\sqrt{13}\mu_1}{2}Z_3-\frac{2\mu_1}{2}W_1\\
    &+\frac{\mu_2}{2}Z_2-\frac{\mu_2}{2}W_2\\
    =&\frac{-\lambda(2-\sqrt{13})\mu_3+(2+\sqrt{13})\mu_3-(2+\sqrt{13})\mu_1}{2}W_1\\
    &+\frac{\lambda(2+\sqrt{13})\mu_2+(2-\sqrt{13})\mu_2-(2-\sqrt{13})\mu_1}{2}W_3\\
    &+\frac{\lambda(2+\sqrt{13})\mu_2+(2-\sqrt{13})\mu_2-(2-\sqrt{13})\mu_1}{2}Z_1\\
    &+\frac{\lambda(2-\sqrt{13})\mu_3-(2+\sqrt{13})\mu_3+(2+\sqrt{13})\mu_1}{2}Z_3.
\end{align*}
This implies
\begin{align*}
    &\left\{\begin{array}{l}\lambda(2+\sqrt{13})\mu_2+(2-\sqrt{13})\mu_2-(2-\sqrt{13})\mu_1=0\\
    \lambda(2-\sqrt{13})\mu_3-(2+\sqrt{13})\mu_3+(2+\sqrt{13})\mu_1=0
    \end{array}\right.\\
    \Longrightarrow&\left\{\begin{array}{l}\lambda(2+\sqrt{13})(2-\sqrt{13})\mu_2\mu_3+(2-\sqrt{13})^2\mu_2\mu_3-(2-\sqrt{13})^2\mu_1\mu_3=0\\
    \lambda(2+\sqrt{13})(2-\sqrt{13})\mu_2\mu_3-(2+\sqrt{13})^2\mu_2\mu_3+(2+\sqrt{13})^2\mu_1\mu_2=0
    \end{array}\right.\\
    \Longrightarrow&\left((2-\sqrt{13})^2+(2+\sqrt{13})^2\right)\mu_2\mu_3-\left((2+\sqrt{13})^2\mu_2+(2-\sqrt{13})^2\mu_3\right)\mu_1=0\\
    \Longrightarrow&\mu_1=\frac{34\mu_2\mu_3}{(2+\sqrt{13})^2\mu_2+(2-\sqrt{13})^2\mu_3}.    
\end{align*}
Conversely, if $\mu_1=\frac{34\mu_2\mu_3}{(2+\sqrt{13})^2\mu_2+(2-\sqrt{13})^2\mu_3},$ then, given $$X=x_1(\sqrt{13}W_2+2Z_2)+x_2(W_1+Z_3)+x_3(W_3+Z_1)+x_4(W_1-Z_3)+x_5(W_3-Z_1)\in\mathfrak{m}_{\{\alpha_2\}}$$ we have
\begin{align*}
    A(\sqrt{13}W_2+2Z_2+W_1+Z_3)=&\sqrt{13}\mu_1x_1W_2+2\mu_1x_1Z_2+\mu_2x_2(W_1+Z_3)\\
    &+\mu_2x_3(W_3+Z_1)+\mu_3x_4(W_1-Z_3)+\mu_3x_5(W_3-Z_1).
\end{align*}
For $\lambda\in\mathbb{R}$ we can compute \eqref{auxiliar:2} as follows:
\begin{align*}
    &[\lambda\left(\sqrt{13}Z_2-2W_2\right)+X,AX]\\
    =&\lambda[\sqrt{13}Z_2-2W_2,AX]+[X,AX]\\
    =&\lambda\{\sqrt{13}\mu_2x_2[Z_2,W_1]+\sqrt{13}\mu_2x_2[Z_2,Z_3]+\sqrt{13}\mu_2x_3[Z_2,W_3]+\sqrt{13}\mu_2x_3[Z_2,Z_1]\\
    &+\sqrt{13}\mu_3x_4[Z_2,W_1]-\sqrt{13}\mu_3x_4[Z_2,Z_3]+\sqrt{13}\mu_3x_5[Z_2,W_3]-\sqrt{13}\mu_3x_5[Z_2,Z_1]\\
    &-2\mu_2x_2[W_2,W_1]-2\mu_2x_2[W_2,Z_3]-2\mu_2x_3[W_2,W_3]-2\mu_2x_3[W_2,Z_1]\\
    &-2\mu_3x_4[W_2,W_1]+2\mu_3x_4[W_2,Z_3]-2\mu_3x_5[W_2,W_3]+2\mu_3x_5[W_2,Z_1]\}\\
    &+\sqrt{13}\mu_2x_1x_2[W_2,W_1]+\sqrt{13}\mu_2x_1x_2[W_2,Z_3]+\sqrt{13}\mu_2x_1x_3[W_2,W_3]\\
    &+\sqrt{13}\mu_2x_1x_3[W_2,Z_1]+\sqrt{13}\mu_3x_1x_4[W_2,W_1]-\sqrt{13}\mu_3x_1x_4[W_2,Z_3]\\
    &+\sqrt{13}\mu_3x_1x_5[W_2,W_3]-\sqrt{13}\mu_3x_1x_5[W_2,Z_1]+2\mu_2x_1x_2[Z_2,W_1]\\
    &+2\mu_2x_1x_2[Z_2,Z_3]+2\mu_2x_1x_3[Z_2,W_3]+2\mu_2x_1x_3[Z_2,Z_1]+2\mu_3x_1x_4[Z_2,W_1]\\
    &-2\mu_3x_1x_4[Z_2,Z_3]+2\mu_3x_1x_5[Z_2,W_3]-2\mu_3x_1x_5[Z_2,Z_1]+\sqrt{13}\mu_1x_1x_2[W_1,W_2]\\
    &+2\mu_1x_1x_2[W_1,Z_2]+\mu_2x_2x_3[W_1,W_3]+\mu_2x_2x_3[W_1,Z_1]+\mu_3x_2x_5[W_1,W_3]\\
    &-\mu_3x_2x_5[W_1,Z_1]+\sqrt{13}\mu_1x_1x_2[Z_3,W_2]+2\mu_1x_1x_2[Z_3,Z_2]+\mu_2x_2x_3[Z_3,W_3]\\
    &+\mu_2x_2x_3[Z_3,Z_1]+\mu_3x_2x_5[Z_3,W_3]-\mu_3x_2x_5[Z_3,Z_1]+\sqrt{13}\mu_1x_1x_3[W_3,W_2]\\
    &+2\mu_1x_1x_3[W_3,Z_2]+\mu_2x_2x_3[W_3,W_1]+\mu_2x_2x_3[W_3,Z_3]+\mu_3x_3x_4[W_3,W_1]\\
    &-\mu_3x_3x_4[W_3,Z_3]+\sqrt{13}\mu_1x_1x_3[Z_1,W_2]+2\mu_1x_1x_3[Z_1,Z_2]+\mu_2x_2x_3[Z_1,W_1]\\
    &+\mu_2x_2x_3[Z_1,Z_3]+\mu_3x_3x_4[Z_1,W_1]-\mu_3x_3x_4[Z_1,Z_3]+\sqrt{13}\mu_1x_1x_4[W_1,W_2]\\
    &+2\mu_1x_1x_4[W_1,Z_2]+\mu_2x_3x_4[W_1,W_3]+\mu_2x_3x_4[W_1,Z_1]+\mu_3x_4x_5[W_1,W_3]\\
    &-\mu_3x_4x_5[W_1,Z_1]-\sqrt{13}\mu_1x_1x_4[Z_3,W_2]-2\mu_1x_1x_4[Z_3,Z_2]-\mu_2x_3x_4[Z_3,W_3]\\
    &-\mu_2x_3x_4[Z_3,Z_1]-\mu_3x_4x_5[Z_3,W_3]+\mu_3x_4x_5[Z_3,Z_1]+\sqrt{13}\mu_1x_1x_5[W_3,W_2]\\
    &+2\mu_1x_1x_5[W_3,Z_2]+\mu_2x_2x_5[W_3,W_1]+\mu_2x_2x_5[W_3,Z_3]+\mu_3x_4x_5[W_3,W_1]\\
    &-\mu_3x_4x_5[W_3,Z_3]-\sqrt{13}\mu_1x_1x_5[Z_1,W_2]-2\mu_1x_1x_5[Z_1,Z_2]-\mu_2x_2x_5[Z_1,W_1]\\
    &-\mu_2x_2x_5[Z_1,Z_3]-\mu_3x_4x_5[Z_1,W_1]+\mu_3x_4x_5[Z_1,Z_3]\\
    =&\frac{\lambda}{2}\left\{\left((\sqrt{13}+2)\mu_2x_2+(\sqrt{13}-2)\mu_3x_4\right)Z_1\right.\\
    &+\left((\sqrt{13}+2)\mu_2x_2-(\sqrt{13}-2)\mu_3x_4\right)W_3\\
    &-\left((\sqrt{13}+2)\mu_2x_3+(\sqrt{13}-2)\mu_3x_5\right)Z_3\\
    &\left.-\left((\sqrt{13}+2)\mu_2x_3-(\sqrt{13}-2)\mu_3x_5\right)W_1\right\}\\
    &+x_1\frac{(\sqrt{13}-2)x_2(\mu_1-\mu_2)-(\sqrt{13}+2)x_4(\mu_1-\mu_3)}{2}Z_1\\
    &+x_1\frac{(\sqrt{13}-2)x_2(\mu_1-\mu_2)+(\sqrt{13}+2)x_4(\mu_1-\mu_3)}{2}W_3\\
    &-x_1\frac{(\sqrt{13}-2)x_3(\mu_1-\mu_2)-(\sqrt{13}+2)x_5(\mu_1-\mu_3)}{2}Z_3\\
    &-x_1\frac{(\sqrt{13}-2)x_3(\mu_1-\mu_2)+(\sqrt{13}+2)x_5(\mu_1-\mu_3)}{2}W_1.
    \end{align*}
Now, since $$\mu_1=\frac{34\mu_2\mu_3}{(2+\sqrt{13})^2\mu_2+(2-\sqrt{13})^2\mu_3},$$
then $$\mu_1-\mu_2=-\frac{(2+\sqrt{13})^2(\mu_2-\mu_3)\mu_2}{(2+\sqrt{13})^2\mu_2+(2-\sqrt{13})^2\mu_3},$$ and $$\mu_1-\mu_3=\frac{(2-\sqrt{13})^2(\mu_2-\mu_3)\mu_3}{(2+\sqrt{13})^2\mu_2+(2-\sqrt{13})^2\mu_3}.$$ Hence, for $(i,j)\in\{(2,4),(3,5)\}$
\begin{align*}
    &\frac{(\sqrt{13}-2)x_i(\mu_1-\mu_2)\pm(\sqrt{13}+2)x_j(\mu_1-\mu_3)}{2}\\
    =&\frac{9(\mu_2-\mu_3)}{2((2+\sqrt{13})^2\mu_2+(2-\sqrt{13})^2\mu_3))}\left(-(\sqrt{13}+2)\mu_2x_i\pm(\sqrt{13}-2)\mu_3x_j\right).
\end{align*}
This implies that
\begin{align*}
    &[\lambda\left(\sqrt{13}Z_2-2W_2\right)+X,AX]\\
    =&\frac{1}{2}\left(\lambda-\frac{9x_1(\mu_2-\mu_3)}{(2+\sqrt{13})^2\mu_2+(2-\sqrt{13})^2\mu_3}\right)\left((\sqrt{13}+2)\mu_2x_2+(\sqrt{13}-2)\mu_3x_4\right)Z_1\\
    &+\frac{1}{2}\left(\lambda-\frac{9x_1(\mu_2-\mu_3)}{(2+\sqrt{13})^2\mu_2+(2-\sqrt{13})^2\mu_3}\right)\left((\sqrt{13}+2)\mu_2x_2-(\sqrt{13}-2)\mu_3x_4\right)W_3\\
    &-\frac{1}{2}\left(\lambda-\frac{9x_1(\mu_2-\mu_3)}{(2+\sqrt{13})^2\mu_2+(2-\sqrt{13})^2\mu_3}\right)\left((\sqrt{13}+2)\mu_2x_3+(\sqrt{13}-2)\mu_3x_5\right)Z_3\\
    &-\frac{1}{2}\left(\lambda-\frac{9x_1(\mu_2-\mu_3)}{(2+\sqrt{13})^2\mu_2+(2-\sqrt{13})^2\mu_3}\right)\left((\sqrt{13}+2)\mu_2x_3-(\sqrt{13}-2)\mu_3x_5\right)W_1.
\end{align*}
Therefore, $[\lambda\left(\sqrt{13}Z_2-2W_2\right)+X,AX]=0$ for $\lambda=\frac{9x_1(\mu_2-\mu_3)}{(2+\sqrt{13})^2\mu_2+(2-\sqrt{13})^2\mu_3}.$ This shows that $\langle\cdot,\cdot\rangle$ is a g.o. metric. The proof is complete.
\end{proof}

\subsection{Equigeodesics} Given a vector $X\in\mathfrak{g}$ and a linear subspace $\mathfrak{u}\subseteq\mathfrak{g},$ let $X_{\mathfrak{u}}$ denote the orthogonal projection of $X$ onto $\mathfrak{u}.$ 

\begin{definition}
Let $G$ be a compact Lie group, and $H$ a closed subgroup of $G.$ A vector $X\in\mathfrak{g}$ that is a geodesic vector for any $G$-invariant metric on $G/H$ is called an \textit{equigeodesic vector.}
\end{definition}

The following proposition establishes a characterization of equigeodesic vectors on a homogeneous space under certain conditions. 
\begin{proposition}[\cite{GGEquigeodesics}]\label{GG:proposition} Let $G$ be a compact Lie group, $H$ a closed subgroup of $G$, and $G/H$ a homogeneous space such that every irreducible $H$-submodule of multiplicity greater than one is orthogonal. Then $X\in\mathfrak{m}$ is an equigeodesic vector if and only if \begin{equation}\label{equigeodesic:formula}
    [X,T_i^j(X_{\mathfrak{m}_i})+T_j^i(X_{\mathfrak{m}_j})]_{\mathfrak{m}}=0,\ i,j=1,...,s,
\end{equation} 
 where $\{T_i^j:i,j\in\{1,...,s\}\}$ is a family of linear maps $T_i^j:\mathfrak{m}_i\rightarrow\mathfrak{m}_j$ satisfying the conditions $i)-iv)$ in Section \ref{section:2.1}.   
\end{proposition}
We may use Proposition \ref{GG:proposition} to characterize equigeodesic vectors on flags of $\mathfrak{g}_2.$
\begin{theorem}\label{equigeodesic:theorem} Let $\mathbb{F}_\Theta$ be a flag of $\mathfrak{g}_2.$
\begin{itemize}
    \item[$a)$] A vector $X\in\mathfrak{m}_\emptyset$ is equigeodesic if and only if $$X\in\Span\{W_1,Z_3\}\cup\Span\{W_2,Z_2\}\cup\Span\{W_3,Z_1\}.$$
    
    \item[$b)$] A vector $X\in\mathfrak{m}_{\{\alpha_1\}}$ is equigeodesic if and only if $X\in\Span\{Z_3\}$ or $$X\in\{w_2W_2+w_3W_3+z_1Z_1+z_2Z_2:w_2z_1+w_3z_2=0\}$$
    \item[$c)$] A vector $X\in\mathfrak{m}_{\{\alpha_2\}}$ is equigeodesic if and only if $$X\in\Span\{\sqrt{13}Z_2-2W_2\}\cup\Span\{W_1,W_3,Z_1,Z_3\}.$$
\end{itemize}
\end{theorem}
\begin{proof} Let $\Theta=\emptyset,$ and consider $\left(\mathfrak{m}_\emptyset\right)_i,\left(T_\emptyset\right)_i^j,\ i,j=1,2,3,4,5,6$ as in the proof of Theorem \ref{metrics:theorem}. For each $i\in\{1,2,3,4,5,6\},$ we have that  $\left(\mathfrak{m}_\emptyset\right)_i$ is an orthogonal $H$-module since it is one-dimensional, so we can apply Proposition \ref{GG:proposition}. For a given $$X=\sum\limits_{i=1}^3w_iW_i+\sum\limits_{i=1}^3z_iZ_i\in\mathfrak{m}_\emptyset$$ we have 
\begin{align*}
&(T_\emptyset)_1^2\left(X_{\left(\mathfrak{m}_\emptyset\right)_1}\right)=w_1Z_3,\ (T_\emptyset)_2^1\left(X_{\left(\mathfrak{m}_\emptyset\right)_2}\right)=z_3W_1,\\
&(T_\emptyset)_3^4\left(X_{\left(\mathfrak{m}_\emptyset\right)_3}\right)=w_2Z_2,\ (T_\emptyset)_4^3\left(X_{\left(\mathfrak{m}_\emptyset\right)_4}\right)=z_2W_2,\\
&(T_\emptyset)_5^6\left(X_{\left(\mathfrak{m}_\emptyset\right)_5}\right)=w_3Z_1,\ (T_\emptyset)_6^5\left(X_{\left(\mathfrak{m}_\emptyset\right)_6}\right)=z_1W_3.
\end{align*}
Therefore, $X$ is equigeodesic if and only if the following equations hold:
\begin{align}\label{auxiliar:equi:1}
\begin{split}
    &2[X,w_iW_i]=2[X,z_iZ_i]=0,\ i=1,2,3,\\
    &[X,w_1Z_3+z_3W_1]=[X,w_2Z_2+z_2W_2]=[X,w_3Z_1+z_1W_3]=0.
\end{split}
\end{align}
If $X$ is equigeodesic, then in particular, 
\begin{align*}
    &0=2[X,w_1W_1]=w_1w_3W_2-w_1w_2W_3+w_1z_2Z_1-w_1z_1Z_2,\\
    &0=2[X,w_2W_2]=-w_2w_3W_1+w_1w_2W_3+w_2z_3Z_1-w_2z_1Z_3,\\
    &0=2[X,w_3W_3]=w_2w_3W_1-w_1w_3W_2+w_3z_3Z_2-w_3z_2Z_3,\\
    &0=2[X,z_1Z_1]=-z_1z_2W_1-z_1z_3W_2+w_1z_1Z_2+w_2z_1Z_3,\\
    &0=2[X,z_2Z_2]=z_1z_2W_1-z_2z_3W_3-w_1z_2Z_1+w_3z_2Z_3,\\
    &0=2[X,z_3Z_3]=z_1z_3W_2+z_2z_3W_3-w_2z_3Z_1-w_3z_3Z_2.
\end{align*}
which implies that $w_i,z_i,\ i=1,2,3$ satisfy the following equations:
\begin{align*}
    \left\{\begin{array}{ll}
    w_1w_i=z_3z_i=w_1z_j=z_3w_j=0, &i\neq 1,\ j\neq 3,\\
    w_2w_i=z_2w_i=w_2z_i=z_2z_i=0, &i\neq 2,\\
    w_3w_i=z_1w_i=w_3z_j=z_1z_j=0, &i\neq 3,\ j\neq 1.
    \end{array}\right.
\end{align*}
From these equations, we can deduce the following statements:
\begin{itemize}
    \item If $w_1\neq 0$ or $z_3\neq 0$ then $w_2=w_3=z_1=z_2=0$ which implies $X\in\Span\{W_1,Z_3\}.$
    \item If $w_2\neq 0$ or $z_2\neq 0$ then $w_1=w_3=z_1=z_3=0$ which implies $X\in\Span\{W_2,Z_2\}.$
    \item If $w_3\neq 0$ or $z_1\neq 0$ then $w_1=w_2=z_2=z_3=0$ which implies $X\in\Span\{W_3,Z_1\}.$
\end{itemize}
Hence, $X\in\Span\{W_1,Z_3\}\cup\Span\{W_2,Z_2\}\cup\Span\{W_3,Z_1\}.$ To show that any $X\in\Span\{W_1,Z_3\}\cup\Span\{W_2,Z_2\}\cup\Span\{W_3,Z_1\}$ satisfies the equations \eqref{auxiliar:equi:1} is a straightforward calculation. This proves $a).$\\

For the proof of $b),$ consider $\Theta=\{\alpha_1\},$ $\left(\mathfrak{m}_{\{\alpha_1\}}\right)_i,$ and $\left(T_{\{\alpha_1\}}\right)_i^{j},$ $i,j=1,2,3$ as in the proof of Theorem \ref{metrics:theorem}. The submodules $\left(\mathfrak{m}_{\{\alpha_1\}}\right)_2$ and $\left(\mathfrak{m}_{\{\alpha_1\}}\right)_3$ are orthogonal (see Remark \ref{remark:2.3}). Therefore, due to Proposition \ref{GG:proposition}, a vector $X=w_2W_2+w_3W_3+\sum\limits_{i=1}^3z_iZ_i\in\mathfrak{m}_{\{\alpha_1\}}$ is equigeodesic if and only if 
\begin{align*}\label{auxiliar:equi:2}
        &2[X,z_3Z_3]_{\mathfrak{m}_{\{\alpha_1\}}}=2[X,w_2W_2+w_3W_3]_{\mathfrak{m}_{\{\alpha_1\}}}=2[X,z_1Z_1+z_2Z_2]_{\mathfrak{m}_{\{\alpha_1\}}}=0,\\
        &[X,z_2W_2-z_1W_3-w_3Z_1+w_2Z_2]_{\mathfrak{m}_{\{\alpha_1\}}}=0.
\end{align*}
Now, let us evaluate each of these expressions:
\begin{align*}
    &2[X,z_3Z_3]_{\mathfrak{m}_{\{\alpha_1\}}}=z_1z_3W_2+z_2z_3W_3-w_2z_3Z_1-w_3z_3Z_2,\\
    &2[X,w_2W_2+w_3W_3]_{\mathfrak{m}_{\{\alpha_1\}}}=w_2z_3Z_1+w_3z_3Z_2-(w_2z_1+w_3z_2)Z_3,\\
    &2[X,z_1Z_1+z_2Z_2]_{\mathfrak{m}_{\{\alpha_1\}}}=-z_1z_2W_2-z_2z_3W_3+(w_2z_1+w_3z_2)Z_3,\\
    &[X,z_2W_2-z_1W_3-w_3Z_1+w_2Z_2]_{\mathfrak{m}_{\{\alpha_1\}}}=w_3z_3W_2-w_2Z_3W_3+z_2z_3Z_1-z_1z_3Z_2.
\end{align*}
Therefore, $X$ is equigeodesic if and only if the following system of equations is satisfied
\begin{equation*}
    \left\{\begin{array}{ll}
    z_3w_i=z_3z_j=0,&i\neq 1,\ j\neq 3,\\
    w_2z_1+w_3z_2=0,\end{array}\right.
\end{equation*}
or, equivalently, $w_2=w_3=z_1=z_2=0,$ (in which case $X\in\Span\{Z_3\}$) or $z_3=0,\ w_2z_1+w_3z_2=0.$ This completes the proof of $b).$\\

For the proof of $c),$ as before we are going to consider $(\mathfrak{m}_{\{\alpha_2\}})_i,\ \left(T_{\{\alpha_2\}}\right)_i^j,\ i,j=1,2,3$ as in the proof of Theorem \ref{metrics:theorem}. In this case, the submodules $(\mathfrak{m}_{\{\alpha_2\}})_i$ are all inequivalent, i.e., all of them have multiplicity one, so the hypothesis of Proposition \ref{GG:proposition} hold. Given a vector $$X=x_1(\sqrt{13}W_2+2Z_2)+x_2(W_1+Z_3)+x_3(W_3+Z_1)+x_4(W_1-Z_3)+x_5(W_3-Z_1)\in\mathfrak{m}_{\{\alpha_2\}},$$ the equations \eqref{equigeodesic:formula} are equivalent to
\begin{align*}
    &2[X,x_1(\sqrt{13}W_2+2Z_2)]_{\mathfrak{m}_{\{\alpha_2\}}}=0,\\
    &2[X,x_2(W_1+Z_3)+x_3(W_3+Z_1)]_{\mathfrak{m}_{\{\alpha_2\}}}=0,\\
    &2[X,x_4(W_1-Z_3)+x_5(W_3-Z_1)]_{\mathfrak{m}_{\{\alpha_2\}}}=0.
\end{align*}
By computing these Lie brackets we obtain
\begin{align*}
    2[X,x_1(\sqrt{13}W_2+2Z_2)]_{\mathfrak{m}_{\{\alpha_2\}}}=&-x_1x_2(2-\sqrt{13})(W_3+Z_1)\\
    &+x_1x_3(2-\sqrt{13})(W_1+Z_3)\\
    &+x_1x_4(2+\sqrt{13})(W_3-Z_1)\\
    &-x_1x_5(2+\sqrt{13})(W_1-Z_3)\\
    2[X,x_2(W_1+Z_3)+x_3(W_3+Z_1)]_{\mathfrak{m}_{\{\alpha_2\}}}=&x_1x_2(2-\sqrt{13})(W_3+Z_1)\\
    &-x_1x_3(2-\sqrt{13})(W_1+Z_3)\\
    2[X,x_4(W_1-Z_3)+x_5(W_3-Z_1)]_{\mathfrak{m}_{\{\alpha_2\}}}=&-x_1x_4(2+\sqrt{13})(W_3-Z_1)\\
    &+x_1x_5(2+\sqrt{13})(W_1-Z_3).\\
\end{align*}
Hence, $X$ equigeodesic if and only if $x_1x_i=0,\ i=2,3,4,5;$ that is, $x_1=0,$ or $x_i=0,\ i=2,3,4,5.$ This proves that the space of equigeodesic vectors is \begin{align*}&\Span\{\sqrt{13}Z_2-2W_2\}\cup\Span\{W_1+Z_3,W_3+Z_1,W_1-Z_3,W_3-Z_1\}\\
=&\Span\{\sqrt{13}Z_2-2W_2\}\cup\Span\{W_1,W_3,Z_1,Z_3\}.\end{align*}
The proof is complete.
\end{proof}
\section{The homogeneous Ricci flow}\label{sec:Ricci}
The parametrization of the homogeneous metrics provided by Theorem \ref{metrics:theorem} allows us to deal with invariant geometry of real flag manifolds of type $\mathfrak{g}_2$ in a more practical way. In particular, the set of invariant metrics on the flag $\mathbb{F}_{\{\alpha_2\}}$ can be identified with the open subset $(\mathbb{R}^+)^3$  consisting of all the points in the Euclidean three-dimensional space with positive coordinates.   In this section, we will make a qualitative analysis of the homogeneous Ricci flow on the flag manifold $\mathbb{F}_{\{\alpha_2\}}.$ For this purpose, we recall the following well-known formula for the Ricci curvature on a reductive homogeneous space given in \cite[Corollary 7.38]{Besse}.

\begin{theorem} Let $G$ be a compact Lie group, $H$ a closed subgroup of $G$, and $\mathfrak{g}=\mathfrak{h}\oplus\mathfrak{m}$ a reductive decomposition that is orthogonal with respect to the Killing form $B$ of $\mathfrak{g}.$ For a $\Ad(H)$-invariant inner product $\langle\cdot,\cdot\rangle$ defined on $\mathfrak{m},$ the Ricci tensor associated with $\langle\cdot,\cdot\rangle$ satisfies the following equation:
\begin{align}\label{Ricci:curvature}
    \begin{split}
    \textnormal{Ric}(X,Y)=&-\frac{1}{2}\sum\limits_{i}\langle\left[X,v_i\right]_{\mathfrak{m}},\left[Y,v_i\right]_{\mathfrak{m}}\rangle-\frac{1}{2}B(X,Y)\\
    &+\frac{1}{4}\sum\limits_{i,j}\langle\left[v_i,v_j\right]_{\mathfrak{m}},X\rangle\hspace{0.05cm}\langle\left[v_i,v_j\right]_{\mathfrak{m}},Y\rangle-\langle U(X,Y), Z\rangle,\ X,Y\in\mathfrak{m},
    \end{split}
\end{align}
where $\{v_i\}$ is an $\langle\cdot,\cdot\rangle$-orthonormal basis of $\mathfrak{m},$ $Z=\sum\limits_{i}U(v_i,v_i),$ and $U:\mathfrak{m}\times\mathfrak{m}\rightarrow\mathfrak{m}$ is the linear map defined implicitly by the formula \begin{equation}\label{U:map}2\langle U(u,v),w\rangle=\langle[w,u]_{\mathfrak{m}},v\rangle+\langle[w,v]_{\mathfrak{m}},u\rangle,\ u,v,w\in\mathfrak{m}.\end{equation}    
\end{theorem}
\begin{corollary}\label{Ricci:corollary}
    Consider the flag manifold $\mathbb{F}_{\{\alpha_2\}}.$ Let $\langle\cdot,\cdot\rangle$ be an $\Ad(K_{\{\alpha_2\}})$-invariant inner product on $\mathfrak{m}_{\{\alpha_2\}},$ and $A$ its associated metric operator defined by positive numbers $\mu_1,\mu_2,\mu_3$ as in formula \eqref{metric:alpha_2}. The components of the Ricci tensor corresponding to $\langle\cdot,\cdot\rangle$ with respect to the basis $\mathcal{B}_{\{\alpha_2\}}$ are given by:
    \begin{align*}
    &\textnormal{Ric}_1=\frac{1}{544}\left\{\left(\frac{(\sqrt{13}-2)\mu_1}{\mu_2}\right)^2+\left(\frac{(\sqrt{13}+2)\mu_1}{\mu_3}\right)^2\right\},\\
    &\textnormal{Ric}_2=-\frac{1}{544}\left(\frac{(\sqrt{13}-2)^2\mu_1}{\mu_2}\right)+\frac{1}{2},\\
    &\textnormal{Ric}_2=-\frac{1}{544}\left(\frac{(\sqrt{13}+2)^2\mu_1}{\mu_3}\right)+\frac{1}{2}.
    \end{align*}
\end{corollary}
\begin{proof} For the $\Ad(K_{\{\alpha_2\}})$-invariant metric $\langle\cdot,\cdot\rangle$ determined by $\mu_1,\mu_2,\mu_3>0,$ an $\langle\cdot,\cdot\rangle$-orthonormal basis of $\mathfrak{m}_{\{\alpha_2\}}$ is given by the vectors
    
$$v_1=\frac{\sqrt{13}W_2+2Z_2}{\sqrt{17\mu_1}},\ v_2=\frac{W_1+Z_3}{\sqrt{2\mu_2}},\ v_3=\frac{W_3+Z_1}{\sqrt{2\mu_2}},\ v_4=\frac{W_1-Z_3}{\sqrt{2\mu_3}},\ v_5=\frac{W_3-Z_1}{\sqrt{2\mu_3}}.$$

Additionally, they satisfy the following relations:

\begin{equation}\label{brackets:1}
    \begin{array}{ll}
	[v_1,v_2]_{\mathfrak{m}_{\{\alpha_2\}}}=\frac{2-\sqrt{13}}{2\sqrt{17\mu_1}}v_3, & \textnormal{[}v_1,v_5\textnormal{]}_{\mathfrak{m}_{\{\alpha_2\}}}=\frac{2+\sqrt{13}}{2\sqrt{17\mu_1}}v_4,\\
	\\
	\textnormal{[}v_1,v_3\textnormal{]}_{\mathfrak{m}_{\{\alpha_2\}}}=-\frac{2-\sqrt{13}}{2\sqrt{17\mu_1}}v_2, & [v_2,v_3]_{\mathfrak{m}_{\{\alpha_2\}}}=\frac{(2-\sqrt{13})\sqrt{\mu_1}}{4\sqrt{17}\mu_2}v_1,\\
	\\
	\textnormal{[}v_1,v_4\textnormal{]}_{\mathfrak{m}_{\{\alpha_2\}}}=-\frac{2+\sqrt{13}}{2\sqrt{17\mu_1}}v_5, &  [v_4,v_5]_{\mathfrak{m}_{\{\alpha_2\}}}=-\frac{(2+\sqrt{13})\sqrt{\mu_1}}{4\sqrt{17}\mu_3}v_1.
	\end{array}
\end{equation}\\

Since $[v_i,v_j]$ is always $\langle\cdot,\cdot\rangle$-orthogonal to $v_i$ and $v_j,$ then, by the formula \eqref{U:map},  we have $U\equiv 0.$ Therefore, when applying the equation \eqref{Ricci:curvature} to $v_k,\ k=1,2,3,4,5,$ it simplifies to
\begin{align*}
    \textnormal{Ric}(v_k,v_k)=&-\frac{1}{2}\sum\limits_{i=1}^5\left|\left|[v_k,v_i]_{\mathfrak{m}_{\{\alpha_2\}}}\right|\right|^2-\frac{1}{2}B(v_k,v_k)+\frac{1}{2}\sum\limits_{1\leq i< j\leq 5}\langle[v_i,v_j]_{\mathfrak{m}_{\{\alpha_2\}}},v_k\rangle^2\\
    =&-\frac{1}{2}\sum\limits_{i=1}^5\left|\left|[v_k,v_i]_{\mathfrak{m}_{\{\alpha_2\}}}\right|\right|^2+\frac{1}{2}(v_k,v_k)+\frac{1}{2}\sum\limits_{1\leq i< j\leq 5}\langle[v_i,v_j]_{\mathfrak{m}_{\{\alpha_2\}}},v_k\rangle^2,
\end{align*}
where the norm $||\cdot||$ is taken with respect to $\langle\cdot,\cdot\rangle.$ Computing the formula above we obtain
\begin{align*}
    \textnormal{Ric}(v_1,v_1)=&-\frac{1}{2}\sum\limits_{i=1}^5\left|\left|[v_1,v_i]_{\mathfrak{m}_{\{\alpha_2\}}}\right|\right|^2+\frac{1}{2}(v_1,v_1)+\frac{1}{2}\sum\limits_{1\leq i< j\leq 5}\langle[v_i,v_j]_{\mathfrak{m}_{\{\alpha_2\}}},v_1\rangle^2\\
    =&-\left\{\left(\frac{2-\sqrt{13}}{2\sqrt{17\mu_1}}\right)^2+\left(\frac{2+\sqrt{13}}{2\sqrt{17\mu_1}}\right)^2\right\}+\frac{1}{2\mu_1}\\
    &+\frac{1}{2}\left\{\left(\frac{(2-\sqrt{13})\sqrt{\mu_1}}{4\sqrt{17}\mu_2}\right)^2+\left(-\frac{(2+\sqrt{13})\sqrt{\mu_1}}{4\sqrt{17}\mu_3}\right)^2\right\}\\
    =&-\frac{34}{68\mu_1}+\frac{1}{2\mu_1}+\frac{1}{2}\left\{\frac{(2-\sqrt{13})^2\mu_1}{(16)(17)\mu_2^2}+\frac{(2+\sqrt{13})^2\mu_1}{(16)(17)\mu_3^2}\right\}\\
    =&\frac{\mu_1}{544}\left\{\left(\frac{2-\sqrt{13}}{\mu_2}\right)^2+\left(\frac{2+\sqrt{13}}{\mu_3}\right)^2\right\},\\
    \textnormal{Ric}(v_2,v_2)=&-\frac{1}{2}\sum\limits_{i=1}^5\left|\left|[v_2,v_i]_{\mathfrak{m}_{\{\alpha_2\}}}\right|\right|^2+\frac{1}{2}(v_2,v_2)+\frac{1}{2}\sum\limits_{1\leq i< j\leq 5}\langle[v_i,v_j]_{\mathfrak{m}_{\{\alpha_2\}}},v_2\rangle^2\\
    =&-\frac{1}{2}\left\{\left(\frac{2-\sqrt{13}}{2\sqrt{17\mu_1}}\right)^2+\left(\frac{(2-\sqrt{13})\sqrt{\mu_1}}{4\sqrt{17}\mu_2}\right)^2\right\}+\frac{1}{2\mu_2}\\
    &+\frac{1}{2}\left(-\frac{2-\sqrt{13}}{2\sqrt{17\mu_1}}\right)^2\\
    =&-\frac{\mu_1}{544}\left(\frac{2-\sqrt{13}}{\mu_2}\right)^2+\frac{1}{2\mu_2}=\textnormal{Ric}(v_3,v_3),\\
    \textnormal{Ric}(v_4,v_4)=&-\frac{1}{2}\sum\limits_{i=1}^5\left|\left|[v_4,v_i]_{\mathfrak{m}_{\{\alpha_2\}}}\right|\right|^2+\frac{1}{2}(v_4,v_4)+\frac{1}{2}\sum\limits_{1\leq i< j\leq 5}\langle[v_i,v_j]_{\mathfrak{m}_{\{\alpha_2\}}},v_4\rangle^2\\
    =&-\frac{1}{2}\left\{\left(-\frac{2+\sqrt{13}}{2\sqrt{17\mu_1}}\right)^2+\left(\frac{(2+\sqrt{13})\sqrt{\mu_1}}{4\sqrt{17}\mu_3}\right)^2\right\}+\frac{1}{2\mu_3}\\
    &+\frac{1}{2}\left(\frac{2+\sqrt{13}}{2\sqrt{17\mu_1}}\right)^2\\
    =&-\frac{\mu_1}{544}\left(\frac{2+\sqrt{13}}{\mu_3}\right)^2+\frac{1}{2\mu_3}=\textnormal{Ric}(v_5,v_5).
\end{align*}

The result follows from the fact that $\mathcal{B}_{\{\alpha_2\}}=\left\{\sqrt{\mu_1}v_1,\sqrt{\mu_2}v_2,\sqrt{\mu_2}v_3,\sqrt{\mu_3}v_4,\sqrt{\mu_3}v_5\right\},$ so \begin{align*}
    &\textnormal{Ric}_1=\textnormal{Ric}(\sqrt{\mu_1}v_1,\sqrt{\mu_1}v_1)=\mu_1\textnormal{Ric}(v_1,v_1)\\
    &\textnormal{Ric}_2=\textnormal{Ric}(\sqrt{\mu_2}v_k,\sqrt{\mu_2}v_k)=\mu_2\textnormal{Ric}(v_k,v_k),\ k=2,3,\\
    &\textnormal{Ric}_3=\textnormal{Ric}(\sqrt{\mu_3}v_k,\sqrt{\mu_3}v_k)=\mu_3\textnormal{Ric}(v_k,v_k),\ k=4,5.
\end{align*}
\end{proof}

\begin{theorem} The homogeneous Ricci flow 
	\begin{equation}\label{Ricci:flow}
	\displaystyle\frac{d\langle\cdot,\cdot\rangle}{dt}=-2\textnormal{Ric}
	\end{equation}
	on the flag manifold $\mathbb{F}_{\{\alpha_2\}}$ is equivalent to the autonomous system of ordinary differential equations 
    \begin{equation}\label{main_eq}
        \left\{\begin{array}{rcl}
        x'&=&\dfrac{x}{z}\left(-\dfrac{1}{2}x^2+\dfrac{1}{\sqrt{13}-2}x-\dfrac{1}{4}y^2\right)\vspace{0.3cm}\\
        y'&=&\dfrac{y}{z}\left(-\dfrac{1}{4}x^2+\dfrac{1}{\sqrt{13}+2}y-\dfrac{1}{2}y^2\right)\vspace{0.3cm}\\
        z'&=&-\dfrac{1}{4}(x^2+y^2)
    \end{array}\right.,\ x,y,z>0.
    \end{equation}
\end{theorem}
\begin{proof}
    By Corollary \ref{Ricci:corollary}, the Ricci flow \eqref{Ricci:flow} is equivalent to the system 
    \begin{equation*}
        \left\{\begin{array}{rcl}
        \mu_1'&=&-\dfrac{1}{272}\left\{\left(\dfrac{(\sqrt{13}-2)\mu_1}{\mu_2}\right)^2+\left(\dfrac{(\sqrt{13}+2)\mu_1}{\mu_3}\right)^2\right\},\vspace{0.3cm}\\
        \mu_2'&=&\dfrac{1}{272}\left(\dfrac{(\sqrt{13}-2)^2\mu_1}{\mu_2}\right)-1,\vspace{0.3cm}\\
        \mu_3'&=&\dfrac{1}{272}\left(\dfrac{(\sqrt{13}+2)^2\mu_1}{\mu_3}\right)-1,
        \end{array}\right.
    \end{equation*}
    with $\mu_1,\mu_2,\mu_3>0.$ By making the change of variables $(\mu_1,\mu_2,\mu_3)\mapsto (x,y,z)$ (which maps $\left(\mathbb{R}^+\right)^3$ onto itself) defined by
    \begin{equation}\label{Eq. relation xyz}x=\dfrac{(\sqrt{13}-2)\mu_1}{68\mu_2},\ y=\dfrac{(\sqrt{13}+2)\mu_1}{68\mu_3},\ z=\dfrac{\mu_1}{68},
    \end{equation}
    we obtain the result.
\end{proof}
In what follows, we will do a qualitative analysis of the system \eqref{main_eq}. First, let us set $\alpha:=\sqrt{13}-2,$ $\beta:=\sqrt{13}+2$. Performing the time rescaling $t = z\tau$, we obtain the following polynomial system
\begin{equation}\label{main_eq_1}
\left\{\begin{array}{rcl}
\dot{x}&=&x\left(-\dfrac{1}{2}x^2+\dfrac{1}{\alpha}x-\dfrac{1}{4}y^2\right),\vspace{0.3cm}\\
\dot{y}&=&y\left(-\dfrac{1}{4}x^2+\dfrac{1}{\beta}y-\dfrac{1}{2}y^2\right),\vspace{0.3cm}\\
\dot{z}&=&-\dfrac{z}{4}(x^2+y^2),
\end{array}\right.
\end{equation}
where the dot $\cdot$ represents the derivative of the functions $x(\tau)$, $y(\tau)$ and $z(\tau)$ with respect to the real variable $\tau$. Emphasize that systems \eqref{main_eq} and \eqref{main_eq_1} are topologically equivalent for $z>0.$ Furthermore, we will denote by $X=(P^1,P^2,P^3)$ the vector field associated to system \eqref{main_eq_1}. In addition, if $x,y,z\geq 0$, then the equilibrium points of system \eqref{main_eq_1} are given by $q_1=(2/\alpha,0,0)$, $q_2=(0,2/\beta,0),$ $q_3=(0.0521831,0.352931,0)$ and $q_4=(0,0,z),$ where $z>0$ and $q_4$ is a numerical solution of the equation $X\equiv0.$ This leads us to the first result of this section.

\begin{proposition}
    Let $\phi_t:\mathbb{R}^3\rightarrow\mathbb{R}^3$ be the flow associated with system \eqref{main_eq_1} with $x,y,z>0$. Then, $\phi_t$ has no equilibrium points in finite time, that is, there is no $q\in\mathbb{R}^3$ such that $\phi_t(q) = q$ for all $t$.
\end{proposition}

Although none of the equilibrium points are strictly in the first octant, the other points will help us understand the local dynamics of the given system in this region. 

	\begin{minipage}[t]{.55\textwidth}
\raggedright
\begin{figure}[H]
	\begin{center}
		\begin{overpic}[scale=0.43]{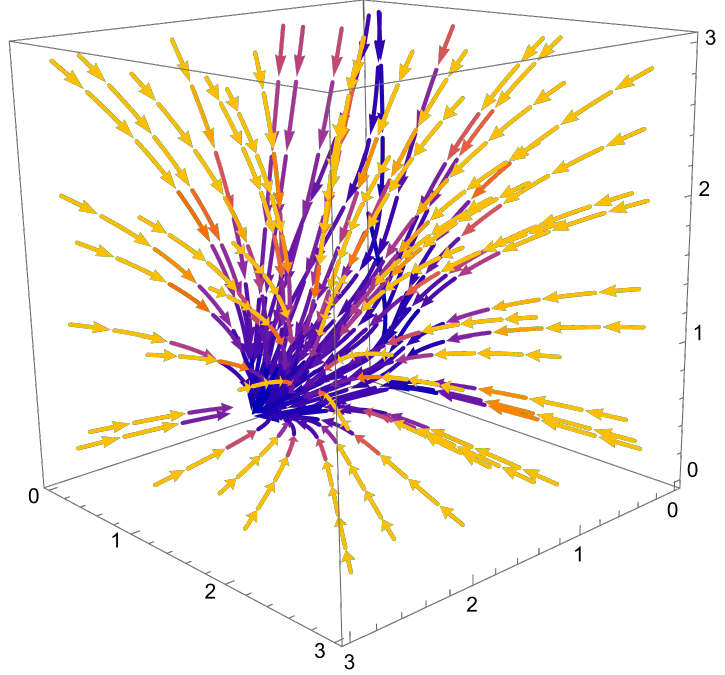}
         \put(99,55){$z$}
		\put(20,12){$y$}
        \put(73,11){$x$}
		\end{overpic}
		\caption{\footnotesize{Phase portrait of system \eqref{main_eq_1}.}}
	\label{FP3D}
	\end{center}
	\end{figure}
\end{minipage}
\begin{minipage}[t]{.55\textwidth}
\raggedright
\begin{figure}[H]
	\begin{center}
		\begin{overpic}[scale=0.42]{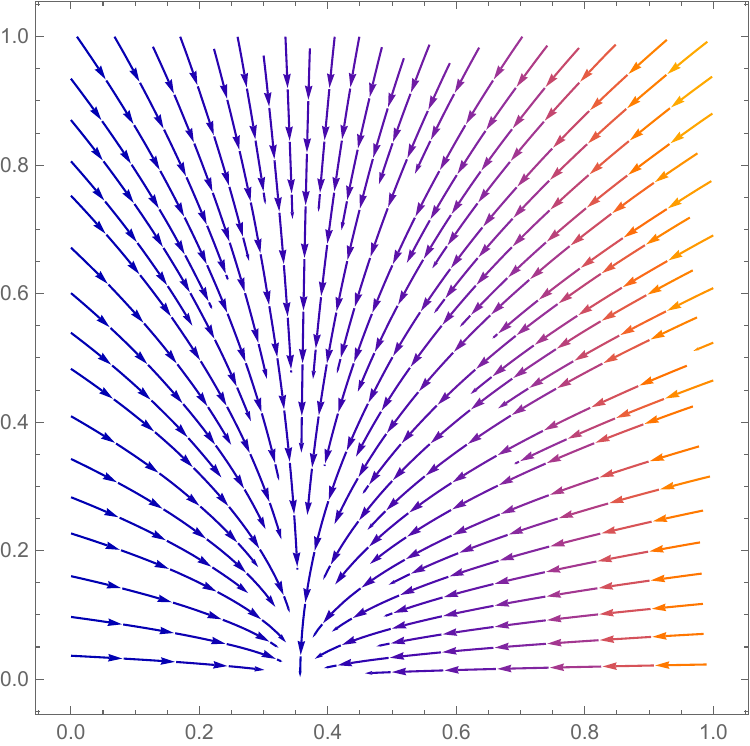}
        \put(1,102){$z$}
		\put(101,2){$y$}
		\end{overpic}
		\caption{\footnotesize{Projection of the phase portrait of system \eqref{main_eq_1} onto the $yz-$plane.}}
	\label{FPYZ}
	\end{center}
	\end{figure}
\end{minipage}
\subsection{Local dynamic}\label{sec_1}
Let us start by classifying all degree 2 invariant surfaces of system \eqref{main_eq_1}. We say that a real polynomial $f=f(x,y,z)$ in the variables $x, y$ and $z$ is a Darboux polynomial of system \eqref{main_eq_1} provided that $(\nabla f)\cdot X=kf,$ where $k=k(x,y,z)$ is a real polynomial of degree at most 2, called the cofactor of $f(x,y,z)$.  If $f(x,y,z)$ is a Darboux polynomial, then the surface $f(x,y,z)=0$ is an invariant algebraic surface, that is, if a orbit of system \eqref{main_eq_1} has a point on this surface, then it is completely contained in it.
\begin{proposition}\label{prop_alg}
    All the invariant algebraic surfaces $f(x, y, z)=0$ of degree 2 of system \eqref{main_eq_1} are given in the following table:
  \begin{table}[h]
	\begin{center}
		\begin{tabular}{| c ||c| c | c |}
			\hline
			 $f(x,y,z)$ & $k(x,y,z)$ \\
			\hline\hline
      $z^2$ & $-\dfrac{x^2}{2}-\dfrac{y^2}{2}$ \\
			\hline
      $x^2$ & $-x^2-\dfrac{y^2}{2}+\dfrac{2x}{\alpha}$ \\
			\hline
      $y^2$ & $-\dfrac{x^2}{2}-y^2+\dfrac{2y}{\beta}$ \\
			\hline
    $xy$ & $-\dfrac{3x^2}{4}-\dfrac{3y^2}{4}+\dfrac{x}{\alpha}+\dfrac{y}{\beta}$ \\
    	\hline
    $xz$ & $-\dfrac{3x^2}{4}-\dfrac{y^2}{2}+\dfrac{x}{\alpha}$ \\
    	\hline
    $yz$ & $-\dfrac{x^2}{2}-\dfrac{3y^2}{4}+\dfrac{y}{\beta}$ \\
			\hline
		\end{tabular}
	\end{center}
	\vspace{0.2cm} \caption{The invariant algebraic surfaces $f(x, y, z)=0$ of degree 2 of system \eqref{main_eq_1}.}\label{tb_2}
\end{table}
\end{proposition}
\begin{proof}
   Substituting
   $$f(z,y,z)=\displaystyle\sum_{i=0}^2\sum_{j=0}^{2-i}\sum_{l=0}^{2-i-j}a_{i,j,k}x^iy^jz^l \quad\text{and}\quad k(z,y,z)=\displaystyle\sum_{i=0}^2\sum_{j=0}^{2-i}\sum_{l=0}^{2-i-j}k_{i,j,k}x^iy^jz^l,$$
      into the equation $(\nabla f)\cdot X=kf$ and using that $X$ is the vector field associated with the system \eqref{main_eq_1}, we obtain, after some tedious calculations, Table \ref{tb_2}. Despite omitting these calculations, the reader can use the functions $f$ and $k$ given in Table \ref{tb_2} and verify that the equation $(\nabla f)\cdot X=kf$ is satisfied.
\end{proof}
From Proposition \ref{prop_alg} we can conclude that the only invariant algebraic surfaces are $x=0,$ $y=0$ and $z=0$. Furthermore, there are no invariant algebraic surfaces of degree 2 for $x,y,z>0.$
 	\begin{minipage}[t]{.55\textwidth}
\raggedright
\begin{figure}[H]
	\begin{center}
		\begin{overpic}[scale=0.43]{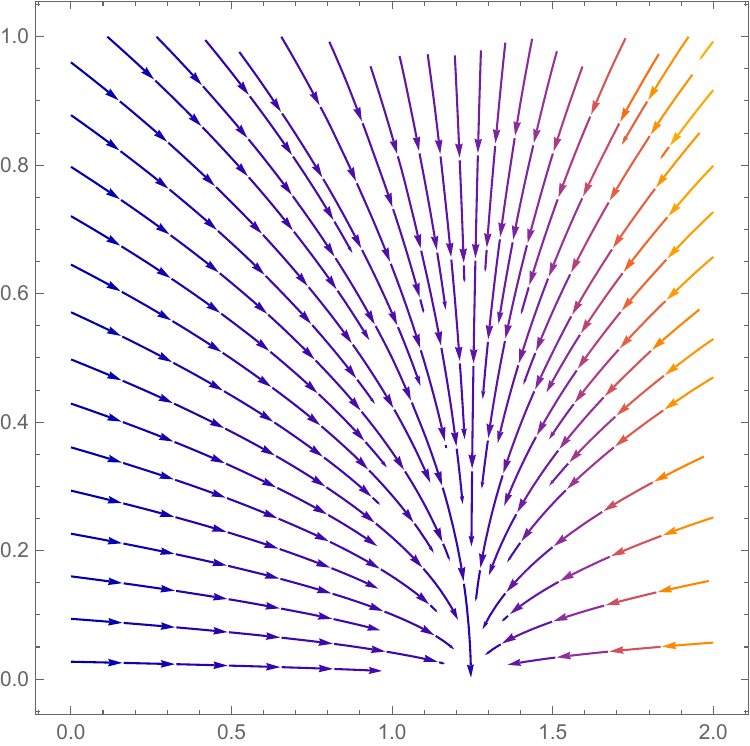}
         \put(1,102){$z$}
		\put(102,2){$x$}
		\end{overpic}
		\caption{\footnotesize{Projection of the phase portrait of system \eqref{main_eq_1} onto the $xz-$plane.}}
	\label{FPXZ}
	\end{center}
	\end{figure}
\end{minipage}
\begin{minipage}[t]{.55\textwidth}
\raggedright
\begin{figure}[H]
	\begin{center}
		\begin{overpic}[scale=0.42]{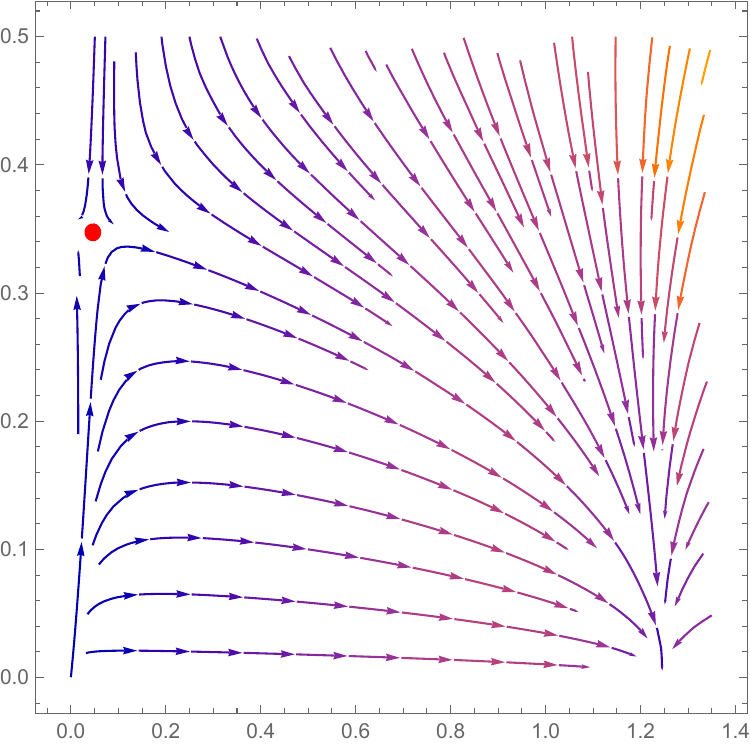}
        \put(1,102){$y$}
		\put(102,2){$x$}
		\end{overpic}
		\caption{\footnotesize{Projection of the phase portrait of system \eqref{main_eq_1} onto the $xy-$plane.}}
	\label{FPXY}
	\end{center}
	\end{figure}
\end{minipage}
\subsubsection{Dynamics around $q_1$}
Here, the Jacobian matrix associated with system \eqref{main_eq_1} at the point $q_1$ has eigenvalues 
$-2/\alpha^2,-1/\alpha^2$ and $-1/\alpha^2$ with corresponding eigenvectors $(1,0,0), (0,0,1)$ and $(0,1,0).$ Therefore, $q_1$ is an attractor point, see Figures \ref{FPXZ} and \ref{FPXY}.
\subsubsection{Dynamics around $q_2$}
Notice that the eigenvalues associated with system \eqref{main_eq_1} at the point $q_2$ are $-2/\beta^2,-1/\beta^2$ and $-1/\beta^2$ with corresponding eigenvectors $(0,1,0), (0,0,1)$ and $(1,0,0).$ Consequantly, $q_2$ is an attractor point, see Figure \ref{FPYZ}.
\subsubsection{Dynamics around $q_3$}
In this case, the eigenvalues associated with $q_3$ are $-0.0625182, -0.0318209$ and $0.0306973$ with eigenvectors $(0.0992779, 0.99506, 0), (0, 0, 1)$ and $(0.99506, -0.0992779, 0),$ respectively. This implies that $q_3$ is a saddle point, see Figure \ref{FPXY}. 
\subsubsection{Dynamics around $q_4$}
Recall that the Jacobian matrix of the vector field associated with the system \eqref{main_eq_1} at the equilibrium point $q_4=(0,0,z)$ is given by
$$\left(\begin{matrix}
0 & 0 & 0\\
0 & 0 & 0\\
0 & 0 & 0\\
\end{matrix}\right).$$
Hence, the equilibrium point $q_4$ has eigenvalues at with real part zero. These types of equilibrium points are known as \textit{nonhyperbolic equilibrium points}. In order to understand the dynamics of the system \eqref{main_eq_1} at $q_4=(0,0,z)$ we must apply the following Blow-up:

$$x=r\widetilde{x}\quad y=r\widetilde{y}\quad z=\widetilde{z}.$$
where $r\in\overline{\mathbb{R}^+}$ and $(\widetilde{x},\widetilde{y},\widetilde{z})\in\mathbb{S}^2$. Roughly speaking, the geometric idea of the blow-up method is to change the equilibrium point $q_4$ by a sphere $\mathbb{S}^2\subset\mathbb{R}^3$, leaving the dynamics away from the $q_4$ unchanged. This allow us to blow-up the dynamics around $q_4$.\\

Now, consider the following chart:
$$\kappa_1: \widetilde{y}=1: x=r_1x_1\quad y=r_1\quad z=z_1.$$

In this chart, the system \eqref{main_eq_1} can be written as:
\begin{equation}\label{main_eq_1_kappa1}
\left\{\begin{array}{rcl}
\dot{x_1}&=&x_1\left(\dfrac{r_1}{4}(1-x_1^2)-\dfrac{1}{\beta}+\dfrac{x_1}{\alpha}\right),\vspace{0.3cm}\\
\dot{r_1}&=&-r_1\left(\dfrac{1}{4}(2+x_1^2)r_1-\dfrac{1}{\beta}\right),\vspace{0.3cm}\\
\dot{z_1}&=&-\dfrac{r_1z_1}{4}(1+x_1^2),
\end{array}\right.
\end{equation}
after desingularization through the division by $r_1$ on the right-hand side. In addition, we denote by $F$ the vector field associated to system \eqref{main_eq_1_kappa1}. In the new coordinates $q_4$ is represented by the equilibrium points 
$$p^+=\left(\dfrac{\alpha}{\beta},0,z_1^*\right)\quad\text{and}\quad p^-=(0,0,z_1^*),$$
where $z_1^*\in\mathbb{R^+}$. Thus, the eigenvalues of the linear part of system \eqref{main_eq_1_kappa1} at the equilibrium point $p^+$ are
$\lambda^+_1=0,\lambda^+_2=1/\beta$ and $\lambda^+_3=1/\beta,$
and the ones of the equilibrium point $p^-$ are $\lambda^-_1=0,\quad \lambda^-_2=1/\beta$ and $\lambda^-_3=-1/\beta.$
This means that $p^\pm$ are nonhyperbolic equilibrium points. 
It is not difficult to find their corresponding eigenvectors $$v^+_1=(0,0,1)\quad\text{and}\quad v^+_2=(1,0,0)$$
and 
$$v^-_1=(0,0,1),\quad v^-_2=\left(0,-\dfrac{4}{z_1^*\beta},1\right) \quad\text{and}\quad v^-_3=(1,0,0),$$
respectively. Nonzero multiples of these eigenvectors are the only eigenvectors of system \eqref{main_eq_1_kappa1} at $p^+$ corresponding to $\lambda_1=0$ and $\lambda_2=\lambda_3=1/\beta$ respectively. Consequently, we must find one generalized eigenvector corresponding to $\lambda_3^+=1/\beta$  and independent of $v_2^+$. Solving the equation $(F-\lambda_3^+I)^2v_3^+=0,$ we get the following generalized eigenvector $$v_3^+=\left(0,1,-\frac{z_1^*(\alpha^2+\beta^2)}{4\beta}\right).$$  

From \textit{center manifold theorem} \cite{10.5555/102732}, we know that there exists an one-dimensional center manifold $W^c(p^\pm)$  tangent to the center subspace $E^c:z_1-$axis of \eqref{main_eq_1_kappa1} at $p^\pm$, there exists a one-dimensional (resp. 2-dimensional) unstable manifold $W^u(p^-)$ (resp $W^u(p^+)$) tangent to the unstable subspace 
$$E^u=\left\{(0,r_1,z_1):r_1=-\dfrac{4z_1}{z_1^*\beta}\right\}\quad \left(\text{resp.}\, E^u=span\left\{v_2^+,v_3^+\right\}\right)$$ of \eqref{main_eq_1_kappa1} at $p^-$ (resp. $p^+$) and there exists a one-dimensional stable manifold $W^s(p^-)$ tangent to the stable subspace $E^s:x_1-$axis of \eqref{main_eq_1_kappa1} at $p^-$. Even more, $W^c(p^\pm)$, $W^s(p^\pm)$ and $W^u(p^\pm)$ are invariant under the flow of \eqref{main_eq_1_kappa1}. The local dynamics of system \eqref{main_eq_1} can be seen in Figure \ref{FP3D}.\\

\subsection{Global dynamic}\label{sec_2}
To investigate the global dynamics of a polynomial differential system in the space $X$, we need to classify the local phase portraits of its finite and infinite equilibrium points on the Poincaré disk.\\

Let $\mathbb{S}^3=\{\mathbf{y}\in\mathbb{R}^4:||\mathbf{y}||=1\}$ be a sphere in $\mathbb{R}^3.$ From \cite{10.2307/2001320}, we know that $X$ induces a vector field in $\mathbb{S}^3,$ which we denote by $p(X)$. The vector field $p(X)$ allows us to study the behavior of $X$ in the neighborhood of infinity, i.e., in the neighborhood of the equator $\mathbb{S}^2=\{\mathbf{y}\in \mathbb{S}^3:y_4=0\}.$ To get the analytical expression for $p(X)$ we shall consider the sphere as a smooth manifold. In this context, it is enough to choose the 3 coordinate neighbourhoods given by $U_i=\{\mathbf{y}\in \mathbb{S}^3: y_i>0\}$, for $i=1, 2, 3$. Denote by $(z_1, z_2, z_3)$ the local coordinates on $U_i$ for $i=1,2,3.$ By \cite{10.2307/2001320}, the vector field $p(X)$ in $U_1$ becomes
 $$\frac{z_3^3}{\Delta(z)^2}(-z_1P^1+P^2,-z_2P^1+P^3,-z_3P^1),$$
where $P^i=P^i(1/z_3,z_1/z_3,z_2/z_3)$ and $\Delta(z)=(1+\sum_{i=1}^3z_i^2)^\frac{1}{2}$. Then, system \eqref{main_eq_1} in the chart $U_1$ is 
\begin{equation}\label{main_eq_1_k1}
\left\{\begin{array}{rcl}
\dot{z_1}&=&z_1\left(\dfrac{1-z_1^2}{4}-\dfrac{z_3}{\alpha}+\dfrac{z_1z_3}{\beta}\right),\vspace{0.3cm}\\
\dot{z_2}&=&z_2\left(\dfrac{1}{4}-\dfrac{z_3}{\alpha}\right),\vspace{0.3cm}\\
\dot{z_3}&=&\dfrac{z_3}{4}\left(2+z_1^2-\dfrac{4z_3}{\alpha}\right),
\end{array}\right.
\end{equation}
We have that in chart $U_1$ system \eqref{main_eq_1_k1} has three equilibrium points:
$$p_1=\left(1,0,0\right),\quad p_2=\left(-1,0,0\right)\quad\text{and}\quad p_3=(0,0,0),$$
The eigenvalues of the linear part of system \eqref{main_eq_1_k1} at the equilibrium points $p_1$ and $p_2$ are $-1/2,1/4$ and $3/4$ with corresponding eigenvectors $(1,0,0),$ $(0,1,0)$ and $(4(\alpha\mp\beta)/5\alpha\beta,0,1)$
and the ones of the origin are $1/4,1/4$ and $1/2$ with corresponding eigenvectors $(1,0,0),$ $(0,1,0)$ and $(0,0,1)$. This implies that the origin is a source and $p_1,p_2$ are saddle points.\\

Likewise, we know that the expression for $p(X)$ in $U_2$ is  given by
$$\frac{z_3^3}{\Delta(z)^2}(-z_1P^2+P^1,-z_2P^2+P^3,-z_3P^2),$$
where $P^i=P^i(z_1/z_3,1/z_3,z_2/z_3)$ and $\Delta(z)=(1+\sum_{i=1}^3z_i^2)^\frac{1}{2}.$ Then, system \eqref{main_eq_1} in the chart $U_2$ is 
\begin{equation}\label{main_eq_1_k2}
\left\{\begin{array}{rcl}
\dot{z_1}&=&z_1\left(\dfrac{1-z_1^2}{4}-\dfrac{z_3}{\beta}+\dfrac{z_1z_3}{\alpha}\right),\vspace{0.3cm}\\
\dot{z_2}&=&z_2\left(\dfrac{1}{4}-\dfrac{z_3}{\beta}\right),\vspace{0.3cm}\\
\dot{z_3}&=&\dfrac{z_3}{4}\left(2+z_1^2-\dfrac{4z_3}{\beta}\right),
\end{array}\right.
\end{equation}
We get that in chart $U_2$ system \eqref{main_eq_1_k2} has three equilibrium points:
$$p_1=\left(1,0,0\right),\quad p_2=\left(-1,0,0\right)\quad\text{and}\quad p_3=(0,0,0),$$
The eigenvalues of the linear part of system \eqref{main_eq_1_k2} at the equilibrium points $p_1$ and $p_2$ are $-1/2,1/4$ and $3/4$
with corresponding eigenvectors $(1,0,0),$ $(0,1,0)$ and $(\mp4(\alpha\mp\beta)/5\alpha\beta,0,1)$ and the ones of the equilibrium point $p_3$ are $1/4,1/4$ and $1/2$ with corresponding eigenvectors $(1,0,0),$ $(0,1,0)$ and $(0,0,1)$. This implies that the origin is a source and $p_1,p_2$ are saddle points.\\

Now, the expression for $p(X)$ in $U_3$ is given by
$$\frac{z_3^3}{\Delta(z)^2}(-z_1P^3+P^1,-z_2P^3+P^2,-z_3P^3),$$
where $P^i=P^i(z_1/z_3,z_2/z_3,1/z_3)$ and $\Delta(z)=(1+\sum_{i=1}^3z_i^2)^\frac{1}{2},$
Then, system \eqref{main_eq_1} in the chart $U_3$ is 
\begin{equation}\label{main_eq_1_k3}
\left\{\begin{array}{rcl}
\dot{z_1}&=&z_1^2\left(-\dfrac{z_1}{4}+\dfrac{z_3}{\alpha}\right),\vspace{0.3cm}\\
\dot{z_2}&=&-z_2^2\left(\dfrac{z_2}{4}-\dfrac{z_3}{\beta}\right),\vspace{0.3cm}\\
\dot{z_3}&=&\dfrac{z_3}{4}\left(z_1^2+z_2^2\right),
\end{array}\right.
\end{equation}
We get that in chart $U_3$ system \eqref{main_eq_1_k3} has the origin as a unique equilibrium point. In addition, the origin is a linearly zero equilibrium (i.e. the Jacobian matrix at $(0,0,0)$ is identically zero). From \textit{center manifold theorem} \cite{10.5555/102732}, we know that there exists an 3-dimensional center manifold $W^c(0,0,0)$  tangent to the center subspace $E^c=\mathbb{R}^3$ of \eqref{main_eq_1_k3} at the origin. Even more, $W^c(0,0,0)$, is invariant under the flow of \eqref{main_eq_1_k3}.\\


Let us denote by $\pi :\mathbb{R}^3\rightarrow\mathbb{R}^3$ the map $\pi(x_1, x_2, x_3) = \frac{1}{\Delta(x)}(x_1, x_2, x_3)$, which shrinks $\mathbb{R}^3$ to its unitary ball and takes the infinity to the sphere $\mathbb{S}^2$. \\

Now, denote $\eta_i$ the point in $\pi(\mathbb{R}^3)$ corresponding to $p_i$ for $i=1,2,3$. Using the information given in the charts $U_i$ for $i=1,2,3$ we have that there are no $\eta_j$ that are in the first octant. This completes the qualitative behavior of differential system \eqref{main_eq_1}.\\

It is worth pointing out that the previous analysis was conducted for the parameters $x, y, z$ defined in \eqref{Eq. relation xyz}. However, this is sufficient to understand and deduce the behavior and properties of the Ricci flow by considering the original parametrization $(\mu_1, \mu_2, \mu_3)$ of the invariant metrics. To illustrate this, we will prove the following proposition that concerns the long-time behavior of the solutions in their original parametrization.

\begin{proposition}
 Let  $t\mapsto\langle\cdot,\cdot\rangle_t$ be a solution of the homogeneous Ricci flow \eqref{Ricci:flow}, and $A_t$ the metric operator associated with $\langle\cdot,\cdot\rangle_t.$  Then $\langle\cdot,\cdot\rangle_t$ collapses over time.
\end{proposition}
\begin{proof} Assume that $A_t$ is determined by $\mu_1(t),\mu_2(t),\mu_3(t)>0,$ for all $t$ and notice that  $\mu_1(t), \mu_2(t), \mu_3(t)$ satisfy the relations
  \begin{equation}\label{eq. mu}
    \mu_1(t)= 68 z(t),\quad \mu_2(t)=\frac{\alpha z(t)}{x(t)} \quad\text{and}\quad \mu_3(t)=\frac{\beta z(t)}{y(t)}, 
\end{equation}
with $\alpha=\sqrt{13}-2$ and $\beta=\sqrt{13}+2.$  Emphasize that there exist four equilibrium points denoted as $\{q_i\}_{i=1}^4$ of system \eqref{main_eq_1}. Further, consider $\mathfrak{R}=\{(x,y,z)\in \mathbb{R}^3: \quad x,y,z>0\}$ as the domain where this metric is defined. Within this domain, there are distinct neighborhoods $U_i$ containing $q_i$ for each $i$, such that $U_i \cap U_j = \varnothing$ for all $i\neq j$. 

In what follows, we study the behavior of the metric at each equilibrium point to understand its effects. For the attractor points $q_1$  and $q_2$ we have that any orbit $\gamma\equiv(\gamma_1,\gamma_2,\gamma_3)\subset U_i$ has as $\omega$-limit set the equilibrium point $q_i$. Thus, for $q_1=(2/\alpha,0,0)$, the curve $\gamma_1(t)\rightarrow 2/\alpha$ and $\gamma_j(t)\rightarrow0$ for $j=2,3$ as $t\rightarrow\infty$. Using the relations \eqref{eq. mu} we get that 
\begin{equation}\label{Eq. Mu flow}
    \mu_1(t)=68\gamma_3(t), \quad \mu_2(t)=\alpha \frac{\gamma_3(t)}{\gamma_1(t)}\quad\text{and}\quad\mu_3(t)=\beta\frac{\gamma_3(t)}{\gamma_2(t)} 
\end{equation}
  
concluding that $\mu_1(t), \mu_2(t) \rightarrow 0$ as $t\rightarrow \infty$, thus the metric is collapsing. For $q_2=(0,\frac{2}{\beta},0)$, the curve $\gamma_2(t)\rightarrow 2/\beta$ and $\gamma_j(t)\rightarrow 0$ for $j=1,3$ as $t\rightarrow \infty$. Therefore, using \eqref{Eq. Mu flow}, we have that $\mu_1(t),\mu_3(t)\rightarrow 0$ as $t\rightarrow\infty$, so the metric is collapsing.\\

The point $q_3=(0.0521831,0.352931,0)$ stands as a saddle point. Upon considering equations \eqref{Eq. Mu flow}, it is evident that this point is associated with a trivial metric and that in the unstable subspace $q_3$ behaves like an unstable equilibrium point. Moreover, since $q_3$ is a saddle point there exists a stable subspace in which all orbits converge to $q_3$, then let $\tilde{\gamma}\equiv(\tilde{\gamma_1},\tilde{\gamma_2},\tilde{\gamma_3})$ be a orbit in $U_3$ such that $\tilde{\gamma}(t)\rightarrow q_3$ as $t\rightarrow\infty$. Thus, $\tilde{\gamma}_3(t)\rightarrow 0$, which implies that $\mu_1(t),\mu_2(t),\mu_3(t)\rightarrow 0$ as $t\rightarrow \infty$, concluding that $\langle\cdot,\cdot\rangle_t$ is collapsing.
\end{proof}

\section*{Acknowledgements}
Brian Grajales is supported by São Paulo Research Foundation (FAPESP) grant 2023/04083-0. Gabriel Rondón is supported by São Paulo Research Foundation (FAPESP) grants 2020/06708-9 and 2022/12123-9. Julieth Saavedra is supported by the Instituto Serrapilheira.

\normalem
\bibliographystyle{abbrv}
\bibliography{references}
\end{document}